\documentclass[12pt, reqno]{amsart}
\usepackage{amssymb, amstext, amscd, amsmath, amsthm}
\usepackage{color}
\usepackage{dsfont}
\usepackage{fullpage}
\usepackage{mathtools}
\usepackage{tikz-cd} 
\usepackage{hyperref}
\usepackage[toc,title,page]{appendix}

\usepackage{xy}
\xyoption{all}

\theoremstyle{plain}
\newtheorem{theorem}{Theorem}[section]
\newtheorem{thm}[theorem]{Theorem}

\newtheorem{cor}[theorem]{Corollary}
\newtheorem{prop}[theorem]{Proposition}
\newtheorem{lem}[theorem]{Lemma}
\newtheorem*{theorem*}{Theorem}
\newtheorem*{conjecture*}{Conjecture}
\theoremstyle{definition}
\newtheorem{rem}[theorem]{Remark}
\newtheorem{defn}[theorem]{Definition}

\newtheorem{eg}[theorem]{Example}

\newtheorem{obs}[theorem]{Observation}

\newcommand{\bC}{{\mathbb{C}}}

\newcommand{\bN}{{\mathbb{N}}}

\newcommand{\bQ}{{\mathbb{Q}}}
\newcommand{\bR}{{\mathbb{R}}}

\newcommand{\bT}{{\mathbb{T}}}

\newcommand{\bZ}{{\mathbb{Z}}}

  \newcommand{\A}{{\mathcal{A}}}
  \newcommand{\B}{{\mathcal{B}}}

  \newcommand{\E}{{\mathcal{E}}}
  \newcommand{\F}{{\mathcal{F}}}
  \newcommand{\G}{{\mathcal{G}}}

\newcommand{\M}{{\mathcal{M}}}
  
\renewcommand{\O}{{\mathcal{O}}}
\renewcommand{\P}{{\mathcal{P}}}
  \newcommand{\Q}{{\mathcal{Q}}}

  \newcommand{\T}{{\mathcal{T}}}

  \newcommand{\Z}{{\mathcal{Z}}}

\newcommand{\fs}{{\mathfrak{s}}}
\newcommand{\ft}{{\mathfrak{t}}}

\newcommand{\Be}{{\mathbf{e}}}
\newcommand{\Bf}{{\mathbf{f}}}

\newcommand{\Bm}{{\mathbf{m}}}
\newcommand{\Bn}{{\mathbf{n}}}

\newcommand{\Bt}{{\mathbf{t}}}

\newcommand{\Bx}{{\mathbf{x}}}
\newcommand{\By}{{\mathbf{y}}}

\newcommand{\rC}{{\mathrm{C}}}

\newcommand{\upchi}{{\raise.35ex\hbox{\ensuremath{\chi}}}}

\newcommand{\qfor}{\quad\text{for}\quad}
\newcommand{\qforal}{\quad\text{for all}\quad}

\newcommand{\mS}{\mathsf{S}_{\bN, \mathsf{E}_n}}

\newcommand{\Aut}{\operatorname{Aut}}
\newcommand{\BS}{\operatorname{BS}}

\newcommand{\id}{{\operatorname{id}}}

\newcommand{\ran}{\operatorname{Ran}}
\newcommand{\spn}{\operatorname{span}}

\newcommand{\GBS}{\operatorname{GBS}}
\newcommand{\lcm}{\operatorname{lcm}}

\newcommand{\ca}{\mathrm{C}^*}

\newcommand{\Fn}{\mathbb{F}_n^+}

\newcommand{\mt}{\varnothing}

\newcommand{\ol}{\overline}

\newcommand{\SM}{\mathsf{S}_{\bN, \mathsf{E}_n}}

\newcommand{\SE}{\mathsf{E}_n}

\begin{document}
\title[Rank $\mathsf k$ BS Semigroups]{Semigroups of Self-Similar Actions \\ and Higher rank Baumslag-Solitar Semigroups}
\author[R. Valente]{Robert Valente}
\address{Robert Valente, 
Department of Mathematics $\&$ Statistics, University of Windsor, Windsor, ON
N9B 3P4, CANADA}
\email{valenter@uwindsor.ca}
\author[D. Yang]{Dilian Yang}\thanks{}
\address{Dilian Yang,
Department of Mathematics $\&$ Statistics, University of Windsor, Windsor, ON
N9B 3P4, CANADA}
\email{dyang@uwindsor.ca}

\thanks{}
\thanks{R.V.~was partially supported by Queen Elizabeth II Graduate Scholarship in Science and Technology (QEII-GSST),
and D.Y.~was partially supported by an NSERC Discovery Grant.}

\begin{abstract}
In this paper, we initiate the study of higher rank Baumslag-Solitar semigroups and their related C*-algebras. 
We focus on two rather interesting classes -- one is related to products of odometers and the other is related to 
Furstenberg’s $\times p, \times q$ conjecture. 
For the former class, whose C*-algebras are studied 
in  \cite{LY19$_1$},  
we here characterize the factoriality of the associated von Neumann algebras 
and further determine their types; 
for the latter, we obtain their canonical Cartan subalgebras. 
In the rank 1 case, we study a more general setting which encompasses (single-vertex) generalized Baumslag-Solitar semigroups. 
One of our main tools in this paper is from self-similar higher rank graphs and their C*-algebras. 
\end{abstract}


\maketitle

\section{Introduction} 

A self-similar higher rank graph $(G, \Lambda)$ is a pair, which consists of a group $G$ and a higher rank graph $\Lambda$ such that $G$ acts on $\Lambda$ from the left 
and $\Lambda$ `acts' on $G$ from the right, where these two actions are compatible in an appropriate way.   
After \cite{EP17, LRRW18, Nek04}, self-similar higher rank graphs and their C*-algebras $\O_{G, \Lambda}$ have been systematically studied in \cite{LY19$_2$, LY21, LY212}. 
In particular, in \cite{LY19$_2$}, when $\Lambda$ is strongly connected, we find a canonical Cartan subalgebra of $\O_{G, \Lambda}$
en route to the study of the KMS states of $\O_{G, \Lambda}$. However, to achieve this, $\Lambda$ is required to be \textsf{locally faithful}. The local faithfulness is a key property to 
obtain the main results in \cite{LY19$_2$}. 
Roughly speaking, it guarantees that one could define a periodicity group in a way very similar to higher rank graphs in \cite{DY091, DY09, HLRS15}. 
It turns out that the local faithfulness condition blocks a lot of interesting examples. 
This provides our starting point of this paper -- to explore non-locally faithful self-similar higher rank graphs and their C*-algebras. During the exploration, we find that a particular 
non-locally faithful class is closely related to Baumslag-Solitar (BS) semigroups. Due to the higher rank feature, we call such self-similar higher rank graphs \textit{higher rank BS semigroups}. 
One extreme case of higher rank BS semigroups is about products of odometers (\cite{LY19$_1$}), while the other extreme 
case is surprisingly related to Furstenberg’s $\times p, \times q$ conjecture.

With semigroups mentioned, there is no surprise that $\Lambda$ is assumed to be single-vertex in this paper. Taking rank 1 (i.e., classical) BS semigroups into consideration, we also consider 
$G=\bZ$ only. 
For single-vertex higher rank graphs, they have been systematically studied in the literature. To name just a few, see, for instance, \cite{DPY08}-\cite{DPY10}, \cite{Yan10, Yan12}. Those graphs seem very special, but 
exhibit a lot of interesting properties. Surprisingly, they are also shown to interact intimately with the Yang-Baxter equation (\cite{Yan16}).  
For BS semigroups, they have been attracting increasing attention in Operator Algebras recently. See \cite{BLRS20, CaHR16, HNSY21, Li21, Spi12} and the references therein. 
Those semigroups provide, on one hand, a class of nice examples for some properties \cite{CaHR16, Li21, Spi12}, and on the other hand, some counter-examples for other properties \cite{BLRS20, HNSY21}. 
Our main purpose in this paper is to mingle single-vertex higher rank graphs and BS semigroups. 

The paper is structured as follows. In Section \ref{S:Pre}, some necessary preliminaries are provided. Although most of them are known, Subsection \ref{SS:semiss} is new, where we introduce a notion 
of semigroups from self-similar actions. Those semigroups are different from self-similar semigroups/monoids in \cite{BGN02} and the references therein (Remark \ref{rem:semiss}). 
Since the rank 1 case is studied in a more general setting, we focus on this case in Section \ref{S:R1}. Even in this case, it includes generalized Baumslag-Solitar (GBS) semigroups, and 
BS semigroups as well, as examples. We study the periodicity of the associated self-similar graph, and obtain a canonical Cartan subalgebra of its C*-algebra 
(Propositions \ref{P:Cartan n=1} and \ref{P:Cartan n>1}). The simplicity of the C*-algebra is characterized in terms of the relation between the number of edges and the restriction map; 
and when it is Kirchberg is also described (Theorem \ref{T:rk1Kir}). We turn to higher rank cases in Section \ref{S:rkBS}. We first propose a notion of higher rank BS semigroups (Definition \ref{D:rkBS}). 
We briefly discuss how higher rank BS semigroups are related to Furstenberg’s $\times p, \times q$ conjecture in Subsection \ref{SS:Fur}. 
We then focus on two extreme classes. The first extreme class is about products of odometers studied in \cite{LY19$_1$}; but here, we investigate the associated von Neumann algebra:
Its factoriality is characterized and its type is also determined (Theorem \ref{T:factor}). The second extreme class seems trivial at first sight, but turns out to be intriguing. We exhibit a canonical Cartan in this case, which is generally a proper subalgebra of the cycline algebra (Theorem \ref{T:F'Cargen}). 
We close with computing the spectrum of the fixed point algebra of its gauge action. 
We  hope that we could push Furstenberg’s $\times p, \times q$ conjecture further in this vein in our future studies.

\subsection*{Notation and Conventions}
Given $1\le n\in \bN$, let $[n]:=\{0, 1, \ldots, n-1\}$. For $1\le \mathsf k\in \bN$, let $\mathds{1}_{\mathsf k}:=(1,\ldots, 1)\in \bN^{\mathsf{k}}$. 

We use the multi-index notation: For $\mathbf{q}=(q_1,\ldots, q_{\mathsf k})$ and $\mathbf{p}=(p_1,\ldots, p_\mathsf{k})$ in $\bZ^{\mathsf k}$ with all $p_i\ne 0$, let 
$\mathbf{p}^{\mathbf q}:=\prod_{i=1}^{\mathsf k} p_i^{q_i}$. 

For convenience, sometimes we also let $\bZ=\langle a\rangle$, which is written multiplicatively. 

As with most literatures in Operator Algebras, all semigroups in this paper are assumed to be monoids, unless otherwise specified.

\section{Preliminaries} 
\label{S:Pre}
\subsection{Single-vertex rank $\mathsf k$ graphs}
\label{SS:single}

A countable small category $\Lambda$ is called a \emph{rank $\mathsf k$ graph} (or \textit{$\mathsf k$-graph}) if there exists a functor $d:\Lambda \to \mathbb{N}^{\mathsf k}$ 
satisfying the following unique factorization property: 
For $\mu\in\Lambda, \Bn, \Bm \in \mathbb{N}^k$ with $d(\mu)=\Bn+\Bm$, there exist unique $\beta\in d^{-1}(\Bn)$ and $\alpha\in d^{-1}(\Bm) $ such that $\mu=\beta\alpha$.
 A functor $f:\Lambda_1 \to \Lambda_2$ is called a \emph{graph morphism} if $d_2 \circ f=d_1$.

Let $\Lambda$ be a $\mathsf{k}$-graph and $\Bn\in \bN^k$. Set $\Lambda^\Bn:=d^{-1}(\Bn)$. For $\mu\in \Lambda$, we write $s(\mu)$ and $r(\mu)$ for the source and range of $\mu$, respectively.  
Then $\Lambda$ is said to be \emph{row-finite} if $\vert v\Lambda^{\Bn}\vert<\infty$ for all $v \in \Lambda^0$ and $\Bn \in \mathbb{N}^{\mathsf k}$;
and \emph{source-free} if $v\Lambda^{\Bn} \neq \mt$ for all $v \in \Lambda^0$ and $\Bn \in \mathbb{N}^{\mathsf k}$. 
For more information about $\mathsf k$-graphs, refer to \cite{KP00}. 
In this paper, all $\mathsf k$-graphs are assumed to be row-finite and source-free. Actually, we focus on a special class of rank $\mathsf k$ graphs -- single-vertex rank $\mathsf k$ graphs.

Single-vertex $\mathsf k$-graphs, at first sight, seem to be a very special class of $\mathsf k$-graphs. It turns out that they are a rather intriguing class to study. They have been systematically studied in the literature, e.g., \cite{DPY08, DY091, DY09, DPY10}. There are close connections with the well-known Yang-Baxter equation \cite{Yan16}.

Let $\{\epsilon_1,\ldots, \epsilon_{\mathsf{k}}\}$ be the standard basis of $\bN^{\mathsf{k}}$, 
and $\Lambda$ be a single-vertex rank $\mathsf{k}$ graph. 
For $1\le i\le \mathsf{k}$, write 
$
\Lambda^{\epsilon_i}:=\{\Bx^i_\fs:\fs\in[n_i]\}, 
$
where $n_i=|\Lambda^{\epsilon_i}|$. It follows from the factorization property of $\Lambda$ that,
for $1\le i<j\le \mathsf{k}$, there is a permutation $\theta_{ij}\in S_{n_i\times n_j}$ 
satisfying the following \textit{$\theta$-commutation relations}
\begin{align*}
 \Bx^i_\fs \Bx^j_\ft = \Bx^j_{\ft'} \Bx^i_{\fs'}
 \quad\text{if}\quad
 \theta_{ij}(\fs,\ft) = (\fs',\ft').
\end{align*}
To emphasize $\theta$-commutation relations involved, this  single-vertex $\mathsf k$-graph $\Lambda$ is denoted as $\Lambda_\theta$ in this paper. So $\Lambda_\theta$ is 
the following (unital) semigroup
\begin{align*}
\Lambda_\theta=\big\langle \Bx_\fs^i: \fs\in [n_i],\, 1\le i\le \mathsf{k};\,  \Bx^i_\fs \Bx^j_\ft = \Bx^j_{\ft'} \Bx^i_{\fs'}
 \text{ whenever }
 \theta_{ij}(\fs,\ft) = (\fs',\ft') \big\rangle^+,
\end{align*}
which is also occasionally written as
\begin{align*}
\Lambda_\theta=\big\langle \Bx_\fs^i:  \fs\in [n_i],\ 1\le i\le \mathsf{k}; \, \theta_{ij}, \, 1\le i<j\le \mathsf{k} \big\rangle^+.
\end{align*}
One should notice that $\Lambda_\theta$ has the cancellation property due to the unique factorization property.
It follows from the $\theta$-commutation relations that every element $w\in \Lambda_\theta$ has the normal form
$
w=\Bx_{u_1}^1\cdots \Bx_{u_{\mathsf{k}}}^{\mathsf{k}}
$
for some $\Bx_{u_i}^i\in\Lambda_\theta^{\epsilon_i\bN}$ ($1\le i\le \mathsf{k}$). Here we use
the multi-index notation: $\Bx^i_{u_i}=\Bx^i_{\fs_1}\cdots \Bx^i_{\fs_n}$ if $u_i=\fs_1\cdots\fs_n$ with all $\fs_i$'s in $[n_i]$.

For $\mathsf{k}=2$, every permutation $\theta\in S_{n_1\times n_2}$ determines a single-vertex rank 2 graph. But for $\mathsf{k}\ge 3$, $\theta=\{\theta_{ij}:1\le i<j\le \mathsf{k}\}$ determines a rank
$\mathsf k$ graph if and only if it satisfies a \textit{cubic condition} (see, e.g., \cite{DY09, FS02} for its definition).
This cubic condition exactly provides interplay between $\mathsf k$-graphs and the Yang-Baxer equation (\cite{Yan16}).

Here are some examples of single-vertex $\mathsf k$-graphs which will be used later. 

\begin{eg}(\textsf{Trivial permutation})  
\label{eg:t}
For $1\le i<j\le \mathsf{k}$, let $\theta_{ij}$ be the trivial permutation: $\theta_{ij}(s,t) = (s,t)$ for all $s\in [n_i]$ and $t\in [n_j]$. 
Then clearly $\Lambda_\theta$ is a $\mathsf{k}$-graph for all $\mathsf{k}\ge 1$, which is written as $\Lambda_\id$. 
\end{eg}

\begin{eg}
\label{eg:d}
(\textsf{Division permutation}) Let $\theta_{ij}$ be defined by $\theta_{ij}(s,t) = (s',t')$, where $s'\in [n_i]$ and $t'\in [n_j]$ are the unique integers such that 
$s + tn_i = t' + s'n_j$. One can check that this determines a $\mathsf{k}$-graph for any $\mathsf{k}\ge 1$ (see, e.g., \cite{LY19$_1$}), denoted as 
$\Lambda_{\mathsf{d}}$. 

In particular, if $n_i=n$ for all $1\le i\le \mathsf{k}$, then $\theta$ coincides with the flip commutation relation: $\theta_{ij}(s,t) = (t,s)$.
\end{eg}

\begin{eg}(\textsf{``Trivial" case}) 
\label{eg:n=1}
Let $n_i=1$ for all $1\le i\le \mathsf{k}$. 
Then $\theta_{ij}$ has to be the trivial commutation relation, which 
is the same as the division commutation relation.
 This is a special case of both Examples \ref{eg:t} and \ref{eg:d}. 
 
Very surprisingly, this case is not trivial at all when it is equipped with self-similar actions! It is extremely interesting and related to Furstenberg’s $\times p, \times q$ conjecture. See Section \ref{SS:Fur} below. 
\end{eg}

\subsection{Self-similar single-vertex $\mathsf k$-graph C*-algebras}
\label{S:sssC*}

To unify the treatments of \cite{Kat08, Kat082} and \cite{Nek04, Nek09}, self-similar graphs and their C*-algebras naturally arise in \cite{EP17} and are well studied there. 
Later, they are generalized to higher rank cases in \cite{LY21}, and are further studied in \cite{LY19$_2$, LY212}.

Since this paper mainly focuses on single-vertex $\mathsf k$-graphs, we adapt the notions of \cite{LY19$_2$, LY21} to our setting and simplify them accordingly.

Let $\Lambda_\theta$ be a single-vertex $\mathsf k$-graph. A bijection $\pi:\Lambda_\theta \to \Lambda_\theta$ is called an \emph{automorphism} of $\Lambda_\theta$ if
$\pi$ preserves the degree map $d$.\footnote{In general, an automorphism on $\Lambda_\theta$ is not necessarily a semigroup automorphism on $\Lambda_\theta$, as a semigroup.} 
Denote by $\Aut(\Lambda_\theta)$ the automorphism group of $\Lambda$.

Let $G$ be a (discrete countable) group. We say that \textit{$G$ acts on $\Lambda_\theta$} if there is a group homomorphism $\varphi$ from $G$ to $\Aut(\Lambda_\theta)$.
For $g\in G$ and $\mu\in\Lambda_\theta$, we often simply write $\varphi_g(\mu)$ as $g\cdot \mu$.

\begin{defn}[{\cite[Definition~3.2]{LY19$_2$}}]
\label{D:ss}
Let $\Lambda_\theta$ be a single-vertex $\mathsf k$-graph, $G$ be a group acting on $\Lambda_\theta$, and $G\times \Lambda_\theta \to G$, $(g,\mu)\mapsto g|_\mu$ be a given map. 
Then we call $(G,\Lambda_\theta)$ \emph{a self-similar $\mathsf k$-graph} if the following properties hold true: 
\begin{enumerate}
\item
$g\cdot (\mu\nu)=(g \cdot \mu)(g \vert_\mu \cdot \nu)$ for all $g \in G,\mu,\nu \in \Lambda_\theta$;

\item
$g \vert_v =g$ for all $g \in G,v \in \Lambda_\theta^0$;

\item
$g \vert_{\mu\nu}=g \vert_\mu \vert_\nu$ for all $g \in G,\mu,\nu \in \Lambda_\theta$; 

\item
$1_G \vert_{\mu}=1_G$ for all $\mu \in \Lambda_\theta$;

\item
$(gh)\vert_\mu=g \vert_{h \cdot \mu} h \vert_\mu$ for all $g,h \in G,\mu \in \Lambda_\theta$.
\end{enumerate}
In this case, we also say that $\Lambda_\theta$ is a \textit{self-similar $\mathsf k$-graph over $G$}, and that \textit{$G$ acts on $\Lambda_\theta$ self-similarly}. 
\end{defn}

\begin{defn}
\label{D:pf}
A self-similar $\mathsf k$-graph $(G,\Lambda_\theta)$ is said to 
be \emph{pseudo-free} if $g \cdot \mu=\mu$ and $g \vert_\mu=1_G $ implies $g=1_G$ for all $g\in G$ and $\mu\in \Lambda_\theta$.
\end{defn}

\begin{defn}[{\cite[Definition~3.8]{LY19$_2$}}]\label{D:O}
Let $(G,\Lambda_\theta)$ be a self-similar $\mathsf k$-graph. The \textit{self-similar $\mathsf k$-graph C*-algebra} $\mathcal{O}_{G,\Lambda_\theta}$
is defined to be the universal unital C*-algebra generated by a family of unitaries $\{u_g\}_{g \in G}$ and
a family of isometries $\{s_\mu: \mu\in \Lambda_\theta\}$ satisfying 
\begin{itemize}
\item[(i)]
$u_{gh}=u_g u_h$ for all $g, \ h \in G$;

\item[(ii)]
$s_\mu s_\nu=s_{\mu\nu}$ for all $\mu,\ \nu\in \Lambda_\theta$;

\item[(iii)] 
$\sum\limits_{\mu\in \Lambda_\theta^\Bn} s_\mu s_\mu^*=I$ for all $\Bn\in \bN^{\mathsf k}$;

\item[(iv)]
$u_g s_\mu=s_{g \cdot \mu} u_{g \vert_\mu}$ for all $g \in G$ and $\mu \in \Lambda_\theta$.
\end{itemize}
\end{defn}

Let us record the following result  (\cite[Propositions~3.12 and 5.10]{LY21}), which will be used later without mentioning.

\begin{prop}
\label{P:genO}
Let $(G,\Lambda_\theta)$ be a self-similar $\mathsf k$-graph. 
Then
\begin{enumerate}
\item
 the linear span of $\{s_\mu u_g s_\nu^*: \mu, \nu \in \Lambda_\theta, g \in G\}$ is a dense $*$-subalgebra of $\mathcal{O}_{G,\Lambda_\theta}$;
\item
The $\mathsf k$-graph C*-algebra $\O_{\Lambda_\theta}$ naturally embeds into $\O_{G,\Lambda_\theta}$;
\item 
 $G$ and $\ca(G)$ embed into $\O_{G, \Lambda_\theta}$, provided that $(G, \Lambda_\theta)$ is pseudo-free and $G$ is amenable. 
\end{enumerate}
\end{prop}

As in \cite{LY21}, let $\gamma$ be the gauge action of $\bT^{\mathsf k}$ on $\O_{\bZ, \Lambda_\theta}$: 
\[
\gamma_\Bt(s_\mu u_g s_\nu^*)=\Bt^{d(\mu)-d(\nu)}s_\mu u_g s_\nu^*
\]
for all $\mu, \nu \in \Lambda_\theta$, $g\in \bZ$, and $\Bt\in \bT^{\mathsf{k}}$. The fixed point algebra, $\O_{\bZ, \Lambda_\theta}^\gamma$, of $\gamma$ is generated by the standard generators $s_\mu u_g s_\nu^*$ with $d(\mu)=d(\nu)$. We often write $\F$ to stand for $\O_{\bZ, \Lambda_\theta}^\gamma$. 
More generally, for $\Bn\in \bN^{\mathsf k}$, we define a mapping on $\O_{\bZ,\Lambda_\theta}$ by 
\[
\Phi_\Bn(x) = \int_{\bT^{\mathsf k}} \Bt^{-\Bn}\gamma_\Bt(x)d\Bt \text{ for all }x \in \O_{\bZ,\Lambda_\theta}.
\] 
Note that for $\mu,\nu \in \Lambda_\theta$, $g \in G$ we have
 \[
 \Phi_\Bn(s_\mu u_g s_\nu^*) = 
\begin{cases}
s_\mu u_g s_\nu^* & \text{if } d(\mu)-d(\nu) = \Bn, \\
0 & \text{otherwise.}
\end{cases}  
\]
In particular, $\O_{\bZ, \Lambda_\theta}^\gamma=\ran\Phi_{\mathbf 0}$. Also $\Phi_{\mathbf 0}$ is a faithful conditional expectation from $\O_{\bZ,\Lambda_\theta}$ onto 
$\O_{\bZ, \Lambda_\theta}^\gamma$. 

\smallskip
We end this subsection by briefly recalling the periodicity of $(G, \Lambda_\theta)$. 
Let $(G,\Lambda_\theta)$ be a self-similar $\mathsf k$-graph.
For $\mu,\nu \in \Lambda_\theta, g \in G$, the triple $(\mu,g,\nu)$ is called \textit{cycline} if $\mu(g \cdot x)=\nu x$ for all $x \in s(\nu)\Lambda^\infty$.
Clearly, every triple $(\mu, 1_G, \mu)$ $(\mu \in \Lambda)$ is cycline. Those cycline triples are said to be \textit{trivial}. 
 An infinite path $x \in \Lambda_\theta^\infty$ is said to be \textit{$G$-aperiodic} if, for $g \in G, \mathbf p, \mathbf q \in \mathbb{N}^{\mathsf k}$ with $g \neq 1_G$ or $\mathbf p \neq \mathbf q$, 
 we have $\sigma^{\mathbf p}(x) \neq g \cdot \sigma^{\mathbf q}(x)$; otherwise, $x$ is called \textit{$G$-periodic}.  
 $(G,\Lambda_\theta)$ is said to be \textit{aperiodic} if there exists a $G$-aperiodic path
 $x\in \Lambda_\theta^\infty$; and \textit{periodic} otherwise.

\begin{thm}[Li-Yang \cite{LY19$_2$}]
\label{T:equper}
$(G,\Lambda_\theta)$ is aperiodic $\iff$ all cycline triples are trivial. 
\end{thm}

\subsection{Right LCM semigroup C*-algebras and their boundary quotient C*-algebras}
\label{SS:rLCM}

Let us recall some basics about right LCM semigroups and their C*-algebras from \cite{BRRW14}.
Let $P$ be a discrete left-cancellative semigroup. 
We say $P$ is a \textit{right LCM semigroup} if any two elements $x,y\in P$ with a right common multiple
have a right least common multiple $z\in P$. 
Equivalently, $P$ is right LCM if, for any $x,y\in P$, the intersection $xP\cap yP$ is either empty or equal to $zP$ for some $z\in P$. 

For a right LCM semigroup $P$, its C*-algebra $\ca(P)$ defined in \cite{Li12} can be greatly simplified as follows: 
$\ca(P)$ is the universal C*-algebra generated by isometries $\{v_p: p\in P\}$ and projections $\{e_{pP}: p\in P\}$ satisfying 
\begin{align}
\label{E:C*(P)}
v_pv_q= v_{pq},\ v_p e_{qP} v_p^*=e_{pqP}, \ e_P=1, \ e_\mt=0,\ e_{pP}e_{qP}=e_{pP\cap qP}
\end{align}
for all $p, q\in P$. 

Recall that a subset $F\subseteq P$ is called a \textit{foundation set} if it is finite and
for each $p \in P$ there exists $q\in F$ such that $pP \cap qP\ne \mt$.
Then the \textit{boundary quotient $\Q(P)$ of $\ca(P)$} is the universal C*-algebra generated by isometries $\{v_p: p\in P\}$ and projections $\{e_{pP}: p\in P\}$ satisfying the relations in \eqref{E:C*(P)}
and 
\[
\prod_{p\in F}(1-e_{pP})=0 \text{ for every foundation set $F\subseteq P$}. 
\]

\subsection{Semigroups from self-similar actions}
\label{SS:semiss}

Let $\Lambda_\theta$ be a single-vertex $\mathsf{k}$-graph. By $\Lambda^\epsilon$, we denote the set of all edges of $\Lambda_\theta$:
$\Lambda^\epsilon=\bigcup_{i=1}^{\mathsf{k}}\{e\in \Lambda: d(e)=\epsilon_i\}$. Suppose that $\bZ=\langle a \rangle$ acts on 
$\Lambda_\theta$ self-similarly.  Then one can naturally associate a semigroup to the self-similar 
$\mathsf k$-graph $(\bZ, \Lambda_\theta)$ as follows:
\begin{align}
\label{E:monoidS}
\mathsf{S}_{\bN, \Lambda_\theta}
:=\left\langle a, e: 
\begin{array}{ll}  
a e=a\cdot e a|_{e} &\text{if $e\in \Lambda_\theta^\epsilon$ and $a|_e\ge 0$}, \\
a e  (a|_{e})^{-1}=a\cdot e &\text{if $e\in \Lambda_\theta^\epsilon$ and $a|_e< 0$},
\end{array}
\; e\in\Lambda_\theta^\epsilon
 \right\rangle^+. 
\end{align}
This is the semigroup we focus on in this paper. Because of its importance, it deserves a name. 

\begin{defn}
The semigroup $\mathsf{S}_{\bN, \Lambda_\theta}$ defined in \eqref{E:monoidS} is called the \textit{semigroup of the self-similar $\mathsf k$-graph $(\bZ, \Lambda_\theta)$}. 
\end{defn}

Here are another semigroup and a group which are closely related to the semigroup $\mathsf{S}_{\bN, \Lambda_\theta}$: 
\begin{align*}
\mathsf{S}_{\bZ, \Lambda_\theta}
&:=\langle a, a^{-1}, e: a e=a\cdot e a|_e, \, e\in \Lambda_\theta^\epsilon \rangle^+,\\
\mathsf{G}_{\bZ, \Lambda_\theta}
&:=\langle a, e: a e=a\cdot e a|_e, \, e\in \Lambda_\theta^\epsilon \rangle.
\end{align*}

\begin{rem} Some remarks are in order. 
\label{rem:semiss}
\begin{enumerate}
\item 
We should mention that $\bN$ used in $\mathsf{S}_{\bN, \Lambda_\theta}$ emphasizes that only non-negative integers from $\bZ$ are involved, although the self-similar graph 
$(\bZ, \Lambda_\theta)$ is considered.

\item 
We have intended to call $\mathsf{S}_{\bN, \Lambda_\theta}$ the self-similar monoid/semigroup of $(\bZ, \Lambda_\theta)$. But the term `self-similar monoids/semigroups' is already used 
in the literature for a very different meaning (see, e.g., \cite{BGN02}), and is similar to the notion of groups over self-similar $\mathsf{k}$-graphs given in \cite{LY19$_2$}. 

\item 
At first glance, it seems that the semigroup $\mathsf{S}_{\bN, \Lambda_\theta}$ has been considered in \cite[Section 3]{LY19$_1$}. But one should notice that it is required that the restriction map is surjective in \cite{LY19$_1$}. This is a rather strong condition. Most semigroups $\mathsf{S}_{\bN, \Lambda_\theta}$ studied in this paper are not covered there. 
\end{enumerate}
\end{rem}

\section{Rank 1 case: more than GBS semigroups}
\label{S:R1}

For $1\le n\in \bN$, let $\mathsf{E}_n$ denote the single-vertex (directed) graph with $n$ edges. 
Suppose that $(\bZ, \mathsf E_n)$ is a self-similar graph. 
Assume that the action of $\bZ$ on the edge set $\mathsf{E}_n^1$ has $\kappa$ orbits $\E_i:=\{e_s^i:  a\cdot e_s^i = e_{s+1\!\!\mod n_i}^i, s\in [n_i]\}$
 for each $1\le i\le \kappa$. Thus 
\[
n=\sum_{i=1}^\kappa n_i \quad\text{and}\quad \mathsf{E}_n^1=\bigsqcup_{i=1}^\kappa \E_i. 
\]
For $1\le i\le \kappa$, let 
\[
m_i:=\sum_{s\in [n_i]} a|_{e_s^i} \quad \text{and}\quad m:=\sum\limits_{i=1}^\kappa m_i. 
\]
Clearly the self-similar action of $(\bZ, \mathsf{E}_n)$ induces a self-similar graph $(\bZ, \mathsf{E}_{n_i})$ for each $1\le i\le \kappa$. 
Conversely, if there is a self-similar action $\bZ$ on each $\mathsf{E}_{n_i}$, then these $\kappa$ self-similar graphs $(\bZ, \mathsf{E}_{n_i})$ determine a self-similar graph $(\bZ, \mathsf{E}_n)$. 

So, in the rank 1 case,  one can rewrite 
\begin{align*}
\mathsf{S}_{\bN, \mathsf{E}_n}
:=\left \langle a,  e\in \mathsf{E}_n^1: \!\!
\begin{array}{ll}
 a e=a\cdot e a|_{e} & \text{ if }a|_e\ge 0\\
a e  (a|_{e})^{-1}=a\cdot e & \text{ if }a|_e<0
\end{array}
\right\rangle^+.
 \end{align*}

Before going further, we should mention that \cite{Nek04} also deals with the rank 1 case. But there, in terms of our terminology, the action of $\bZ$ on the infinite path 
space $\mathsf E_n^\infty$ is assumed to be faithful. This in particular implies that the self-similar graph $(\bZ, \mathsf E_n)$ is aperiodic. Thus this eliminates all interesting (periodic) self-similar graphs 
(cf. Propositions \ref{P:gper}, \ref{P:Cartan n=1}, \ref{P:Cartan n>1}, and Theorem \ref{T:rk1Kir} below).

\textsf{Throughout this section, we assume that 
\[
m_i\ne 0 \text{ for all }1\le i\le \kappa.
\tag{\dag}
\]
}
This condition assures that the self-similar graph $(\bZ, \mathsf E_n)$ is pseudo-free (Lemma \ref{L:pse}), which is required in \cite{EP17, LY19$_2$, LY21}.

\begin{rem}
\label{R:emb}
It is worth mentioning that, under the assumptions ($\dag$), $\SM$ is embedded into  the group $\mathsf{G}_{\bZ, \mathsf{E}_n}$ (\cite{Adj66}).
\end{rem}

A special class of self-similar graphs is worth mentioning for later use.  

\begin{eg}[\textsf{$(n,m)$-odometer $\mathsf E(n,m)$}]
\label{eg:E(n,m)}
For $1\le n\in \bN$ and $0\ne m \in \bZ$, an \textit{$(n,m)$-odometer} is a self-similar graph $(\bZ, \SE)$ with the action and restriction given by  
\begin{align*}
a\cdot e_s &=
\begin{cases}
 e_{s+1} & \text{ if } 0\le s<n-1,  \\
 e_0 & \text{ if } s=n-1;
 \end{cases}\\
a|_{e_s} &=
\begin{cases}
0  &\text{ if } 0\le s<n-1 , \\
a^{m} & \text{ if } s=n-1.
\end{cases}
\end{align*}
The $(n,m)$-odometer is denoted as $\mathsf E(n,m)$. 
The case of $m=1$ yields the classical odometers which have been extensively studied in the literature (see, e.g., \cite{Nek05} and the references therein).  
\end{eg}

In the sequel, we provide two examples of important semigroups which can be realized as semigroups of self-similar graphs. 

\begin{eg}[\textsf{BS semigroups}]
\label{eg:BS}
For $1\le n\in \bN$ and $0\ne m \in \bZ$, the \textit{Baumslag-Solitar (BS) semigroup} is
\begin{align*}
\BS^+(n, m)
:=\left\langle a, b\mid\!\!
\begin{array}{ll}
a^n b=b a^{m} &\text{if $m>0$}  \\
a^{n} b a^{-m}=b &\text{if $m<0$}
\end{array}
\right\rangle^+.
\end{align*}
The semigroup $\BS^+(n, m)$ can be realized as the semigroup of an $(n,m)$-odometer. 

From now on, we use the semigroups $\BS^+(n,m)$ and $(n,m)$-odometer interchangeably. 
\end{eg}

\begin{eg}[\textsf{GBS semigroups}]
As the name indicates, this example generalizes BS semigroups in Example \ref{eg:BS}.
Let $1\le \kappa\in \bN \cup\{\infty\}$. 
For $1\le n_i\in \bN$ and $0\ne m_i\in \bZ$ ($1\le i\le \kappa$), the \textit{generalized Baumslag-Solitar (GBS) semigroup} is
\begin{align*}
\GBS_\kappa^+(n_i, m_i)
:=\left\langle a, b_i\mid \!\!
\begin{array}{ll}
a^{n_i} b_i=b_i a^{m_i} &\text{if $m_i>0$},  \\
a^{n_i} b_i a^{-m_i}=b_i &\text{if $m_i<0$},
\end{array}
\ 1\le i\le \kappa\right\rangle^+.
\end{align*}
The GBS semigroup $\GBS_k^+(n_i, m_i)$ can also be realized as the semigroup of a self-similar graph as follows. 
Let $\mathsf{E}$ be the single-vertex directed graph with the edge set $\{e^i_s: 1\le i\le \kappa, s\in [n_i]\}$. 
To each $1\le i\le \kappa$, we associate an $(n_i, m_i)$-odometer. 
Then 
$\GBS_\kappa^+(n_i, m_i)\cong \textsf{S}_{\bN, \mathsf{E}}$.
\end{eg}

Therefore, semigroups of self-similar graphs encompass GBS semigroups. 

\begin{rem} In this remark, let us mention some connections with the literature. 
\begin{enumerate}
\item 
BS semigroups usually provide a nice class of examples or counter-examples for some properties (e.g., \cite{BLRS20, HNSY21}). They have been attracting a lot of operator algebraists' attention recently. 
For instance, in \cite{Spi12}, the boundary quotient of the semigroup C*-algebra
$\BS^+(n,m)$ is first investigated via the C*-algebra for a category of paths. In \cite{CaHR16}, the KMS states of the semigroup C*-algebra of quasi-lattice ordered BS semigroups are studied. 
This is generalized to all BS semigroups later in \cite{BLRS20}. 

\item
Very recently, in \cite{CL23} Chen-Li study the C*-algebras for a class of semigroups, which are graphs of semigroups which are constructed very similarly to graphs of groups in \cite{Ser80}. 
There is some intersection: For instance, their semigroups encompass GBS semigroups. However, theirs do not include all semigroups $\mathsf{S}_{\bZ, \Lambda_\theta}$. Most importantly, 
theirs do not include any `genuine' higher rank BS semigroups studied in Section \ref{S:rkBS} below. 
\end{enumerate}
\end{rem}

\subsection{Some basic properties} 
The two lemmas below will be used frequently. One can prove the first one by simple calculations, and the second one by applying Remark \ref{R:emb}. Their proofs are omitted here. 

\begin{lem}
\label{L:formula}
Let $(\bZ, \SE)$ be a self-similar graph. Then, for $\ell\in \bZ$, $1\le i\le \kappa$, and $p\in [n_i]$, one has 
\begin{itemize}
\item[(i)] 
$a^{\ell n_i+ p}\cdot{e_s^i}=e_{(s+p)\!\!\mod\! n_i}^i$;

\item[(ii)] 
$
a^{\ell n_i+ p}|_{e_s^i}=
\begin{cases}
a^{\ell m_i}  &\text{ if $p=0$},\\
a^{\ell m_i} \prod_{q=0}^{p-1} a|_{e_{(s+q)\!\!\mod\! n_i}^i} & \text{ if $0<p<n_i-1$}.
\end{cases}
$ 
\end{itemize}
\end{lem}

\begin{lem}
\label{L:canform}
Every element $x\in  \mathsf{S}_{\bN, \mathsf{E}_n}$ has a unique representation $x= e_\mu a^\ell$ for some $\mu \in \mathsf{E}_n^*$ and $\ell\in \bZ$. 
\end{lem}

\begin{prop}
\label{P:rLCM}
 $\SM$ is right LCM. 
\end{prop}

\begin{proof}
Consider $e_\mu a^k$ and $e_\nu a^\ell$ in $\SM$. It is not hard to see that they have a right common upper bound if, 
and only if either $e_\mu = e_\nu e_{\tilde \mu}$ for some $e_{\tilde \mu}\in \mathsf E_n^*$ or 
$e_\nu = e_\mu e_{\tilde \nu}$ for some $e_{\tilde \nu}\in \mathsf E_n^*$. WLOG we assume that $e_\nu = e_\mu e_{\tilde \nu}$ for some $e_{\tilde \nu}\in \mathsf E_n^*$. 
Let $e_\alpha:=a^{-k}\cdot e_{\tilde \nu}$. Then one can show the following: If $a^\ell \ge a^k|_{e_\alpha}$ (resp. $a^\ell <  a^k|_{e_\alpha}$), then $e_\nu a^\ell$
(reps. $e_\nu  a^k|_{e_\alpha}$) is a least right common upper bound of $e_\mu a^k$ and $e_\nu a^\ell$ (in $\SM$).  
\end{proof}

\begin{rem}
If $a|_e\ge 0$ for all $e\in \mathsf{E}_n$, then $\SM$ is a Zappa-Sz\'ep product of the semigroups $\bN$ and $\Fn$ (\cite{BRRW14}).
\end{rem}

Since $\SM$ is right LCM, from Subsection \ref{SS:rLCM} and the analysis above, one has the following

\begin{cor}
\label{C:QO}
$
\Q(\SM)\cong\O_{\bZ, \mathsf{E}_n}\cong \Q(\mathsf{S}_{\bZ, \mathsf{E}_n}).
$
\end{cor}

\begin{proof}
By Proposition \ref{P:rLCM}, the sets $\{a\}$ and $\{e_i: i\in [n]\}$ are foundation sets of $\SM$. Then the map 
\[
\Q(\SM)\to\O_{\bZ, \mathsf{E}_n},\ v_{e_\mu a^\ell} \mapsto s_{e_\mu} u_{a^\ell},\ 
e_{e_\mu a^\ell\SM}\mapsto s_{e_\mu} s_{e_\mu}^*
\]
yields an isomorphism.  The proof of $\Q(\mathsf S_{\bZ, \mathsf E_n})\cong\O_{\bZ, \mathsf{E}_n}$ is even simpler. 
\end{proof}

\subsection{Pseudo-freeness of $(\bZ, \SE)$} 

Let $(\bZ, \SE)$ be a self-similar graph satisfying our standing assumption ($\dag$).

\begin{lem}
\label{L:pse}The self-similar graph $(\bZ, \mathsf{E}_n)$ is pseudo-free. 
\end{lem}

\begin{proof}
This follows from Lemma \ref{L:formula}. In fact, 
suppose that $g\cdot \mu=\mu$ and $g|_\mu=0$. If $|\mu|=1$, it then follows from Lemma \ref{L:formula} and the assumption ($\dag$) that $g=0$.
Now suppose that $g\cdot \mu=\mu$ and $g|_\mu=0$ with $|\mu|=k$ imply $g=0$. 
Let $g\cdot(\mu e_s^i)=\mu e_s^i$ and $g|_{\mu e_s^i}=0$ for some edge $e_s^i$ in the $i$-th orbit. Then 
$g\cdot \mu g|_\mu\cdot e_s^i=\mu e_s^i\implies g\cdot \mu=\mu$ and $g|_\mu\cdot e_s^i=e_s^i$.
So the latter implies $g|_\mu=a^{\ell n_i}$ for some $\ell\in\bZ$. But also 
$g|_{\mu e_s^i}=0$ implies that $0=g|_\mu|_{e_s^i}=a^{\ell n_i}|_{e_s^i}=a^{\ell m_i}$. 
Hence $\ell=0$ as $m_i\ne 0$. Therefore $g\cdot \mu=\mu$ and $g|_\mu=0$. By our inductive assumption, we have $g=0$. 
This proves the pseudo-freeness of $(\bZ, \mathsf{E}_n)$. 
\end{proof}

\begin{rem}
Lemma \ref{L:pse} is no longer true if $m_i=0$ for some $1\le i\le \kappa$. 
For example, consider the self-similar graph $(\bZ, \textsf{E}_2)$ with $ae_1=e_2a$ and $ae_2=e_1a^{-1}$. Then 
$a^2e_i=e_i$ for $i=1,2$, and $a^2|_{e_i}=0$. Clearly, this self-similar graph $(\bZ, \mathsf{E}_2)$ is not pseudo-free.
\end{rem}

\subsection{The periodicity of $(\bZ, \SE)$} 
In this subsection, we study the periodicity of self-similar graphs $(\bZ, \mathsf{E}_n)$ in detail. We first analyze the case when $\kappa=1$, and then use it to study the general case. 

Recall that $\kappa$ is the number of orbits of $\bZ$ on $\SE$. 

\subsubsection{The case of $\kappa=1$}

\begin{prop}
\label{P:kappa=1}
If $\kappa=1$, then $(\bZ, \mathsf{E}_n)$ is periodic if and only if $n\!\mid\! m$. 
\end{prop}

\begin{proof}
Suppose that $m=n\ell$ for some $\ell \in \bZ$. 
Let $x=e_{i_1}e_{i_2}\cdots\in \mathsf{E}_n^\infty$ be an infinite path. 
Repeatedly applying Lemma \ref{L:formula} gives 
\[
a^{nk}\cdot x=a^{nk}\cdot e_{i_1} a^{n k}|_{e_{i_1}}\cdot (e_{i_2}\cdots)=e_{i_1} a^{n\ell k}\cdot(e_{i_2}\cdots)=\cdots=x \qforal  k\in \bZ.
\]
This shows that every infinite path $x\in \mathsf{E}_n^\infty$ is $\bZ$-periodic in the sense of \cite{LY21}. So $(\bZ, \mathsf{E}_n)$ is periodic. 

It remains to show that if $n\nmid m$ then $(\bZ,\mathsf{E}_n)$ is aperiodic. To the contrary, assume that $(\bZ, \mathsf{E}_n)$ is periodic.
It follows from \cite[Theorem 3.7]{LY19$_2$} $(\bZ, \mathsf{E}_n)$ has a non-trivial cycline triple $(\mu, g, \nu)$. That is, 
$\mu g\cdot x=\nu x$ for all $x\in \mathsf{E}_n^\infty$ with $g\ne 0$ or $\mu\ne \nu$. 

Case (a): $|\mu|<|\nu|$. Then there is a unique $\nu'\in \mathsf{E}_n^*\setminus \mathsf{E}_n^0$ such that $\nu=\mu\nu'$. Thus $g\cdot x=\nu'x$ and so $g\cdot x(0, |\nu'|)=\nu'$
 for all $x\in \mathsf{E}_n^\infty$. This is impossible by noticing that $n$ has to be greater than $1$. 
 
 Case (b): $|\mu|>|\nu|$. Since $x\in \mathsf{E}_n^\infty$ is arbitrary, we replace $x$ with $g^{-1}\cdot x$ and then apply Case (a). 
 
 Case (c): $|\mu|=|\nu|$. Then $\mu=\nu$ and $g\cdot x=x$ for all $x\in \mathsf{E}_n^\infty$. Let $d:=\gcd(m,n)>0$. Write $m=dm_0$ and $n=dn_0$. 
 Since $n\nmid m$, we have $n_0\nmid m_0$. 
Write $x=e_{i_1}e_{i_2}\cdots$ with $i_j\in [n]$. Then 
\[
g\cdot (e_{i_1}e_{i_2}\cdots)=e_{i_1}e_{i_2}\cdots\implies 
g\cdot e_{i_1}=e_{i_1}, \ g|_{e_{i_1}}\cdot e_{i_2}=e_{i_2},\ \ldots
\]
Hence there is a sequence $\{k_i\}_{i\ge 1}\subseteq \bZ$ of (non-zero) integers such that 
\[
g=a^{k_1 n}, \ a^{k_1 m}=a^{k_2 n},\  a^{k_2 m}=a^{k_3 n}, \ a^{k_3 m}=a^{k_4 n}, \ \ldots.
\]
So 
\[
k_1 m = k_2 n,\  k_2 m=k_3n, \ldots
\]
imply
\[
k_1 m_0 = k_2 n_0,\  k_2 m_0=k_3n_0, \ldots.
\]
Thus one has that $m_0^p k_1 = n_0^p k_{p+1}$ for \textit{all} $p\ge 1$. In particular $n_0^p\!\mid\! k_1$ for \textit{all} $p\ge 1$ as $\gcd(m_0, n_0)=1$. 
But $n\!\nmid\! m$ implies $n_0>1$. So $k_1=0$ and hence $g=0$. Then $(\mu, g, \nu)=(\mu, 0, \mu)$ is a trivial cycline triple. This is a contradiction. 
\end{proof}

\begin{rem}
\label{R:Per}
Let $(G, \Lambda)$ be a self-similar $\mathsf{k}$-graph. 
As mentioned in \cite{LY21}, it is easy to see that if $\Lambda$ is periodic, then $(G, \Lambda)$ is periodic. But the converse is not true. 
Here is a class of counter-examples: It is well-known that $\SE$ is aperiodic if $n>1$. But Proposition \ref{P:kappa=1} shows that $(\bZ, \SE)$ is periodic whenever $n\!\mid\! m$.
Therefore, the periodicity of $(G, \Lambda)$ is more complicated than that of the ambient graph $\Lambda$. 
\end{rem}

We now determine all cycline triples of $(\bZ, \SE)$ when $\kappa=1$. Notice that all cycline triples are trivial if $(\bZ, \SE)$ is aperiodic by Theorem \ref{T:equper}. 
Hence it suffices to consider periodic self-similar graphs $(\bZ, \SE)$. 

When $n=1$, there is a unique infinite path. So the following is straightforward. 

\begin{lem}
\label{L:per n=1}
If $n=1$, then every triple $(\mu, a^\ell, \nu)$ is cycline. 
\end{lem}

\begin{prop}
\label{P:kappa1per}
Suppose that $\kappa=1$, $n>1$, and $(\bZ, \SE)$ is periodic. Then $(\mu, g, \nu)$ is cycline if and only if $\mu=\nu$ and $g=a^{\ell n}$ for some $\ell \in \bZ$. 
\end{prop}

\begin{proof}
``If" part is clear. 
For the ``Only if" part, assume that $(\mu, g, \nu)$ is cycline. Then 
\[
\mu g\cdot x= \nu x \text{ for all }x\in \SE^\infty. 
\]
If $|\mu|=|\nu|$, then $\mu=\nu$ and $g\cdot x=x$ for all $x\in \SE^\infty$. As in the proof of Proposition \ref{P:kappa=1}, one can see that $g=a^{\ell n}$ for some $\ell\in \bZ$. 

If $|\mu|\ne |\nu|$, WLOG, $|\nu|> |\mu|$. Then $\nu=\mu\nu'$ for some $\nu'\in \SE^*\setminus\SE^0$ and $g\cdot x= \nu' x$. This is impossible as $n>1$.
\end{proof}

Combining \cite[Theorem 6.6, Theorem 6.13]{LY21} with Corollary \ref{C:QO} yields 

\begin{thm}
\label{T:k=1}
$\Q(\mS)$ with $\kappa=1$ satisfies UCT. It is simple iff $n\!\nmid\! m$. So it is a Kirchberg algebra iff $n\!\nmid\! m$. 
\end{thm}

\subsubsection{The general case}

For the general $\kappa\ge 1$, we begin with a relation between the periodicity of $(\bZ, \SE)$ and that of its restrictions on orbits. 

For simplification, let $\mathfrak N:=\lcm(n_i:1\le i\le \kappa)$. 

 \begin{prop}
 \label{P:gper}
 $(\bZ, \mathsf{E}_n)$ is periodic, if and only if the restriction $(\bZ, \mathsf{E}_{n_i})$ is periodic for each $1\le i\le \kappa$, if and only if $n_i\!\mid\! m_i$ for every $1\le i\le \kappa$. 
 \end{prop}

 \begin{proof}
  If there is $1\le i\le \kappa$ such that $n_i\!\nmid\! m_i$, then $(\bZ, \mathsf{E}_{n_i})$ is aperiodic by Proposition \ref{P:kappa1per}. Then clearly
  $(\bZ, \mathsf{E}_{n})$ is aperiodic.  
 
  Now let us assume that $n_i\!\mid\! m_i$ for all $1\le i\le \kappa$. Say $m_i=n_i \widetilde{m}_i$ with $0\ne \widetilde m_i\in\bZ$ for $1\le i\le \kappa$. 
  Then 
  \[
  a^{\mathfrak N} \, e_{s_1}^{i_1}\cdots e_{s_p}^{i_p}=e_{s_1}^{i_1}\cdots e_{s_p}^{i_p} \, a^{\mathfrak N\, \widetilde m_{i_1}\cdots\widetilde m_{i_p}}.
  \]
  Thus one can check that for arbitrary $x\in \mathsf{E}_n^\infty$ one has 
  $a^{\mathfrak N} \cdot x= x$. Therefore, every infinite path $x$ is $\bZ$-periodic, and so $(\bZ, \mathsf{E}_n)$ is periodic. 
 \end{proof}

We now determine all cycline triples. If $n=1$, this is provided in Lemma \ref{L:per n=1}. 

\begin{prop}
\label{P:cycline}
If $(\bZ, \SE)$ is periodic with $n>1$, then $(\mu, g, \nu)$ is cycline if and only if $\mu=\nu$ and $g=a^{\ell \mathfrak N}$ for some $\ell \in \bZ$. 
\end{prop}

\begin{proof}
``If" part is clear. For the ``Only if" part, assume that $(\mu, g, \nu)$ is cycline. Then 
\[
\mu g\cdot x= \nu x \text{ for all }x\in \SE^\infty. 
\]
If $|\mu|=|\nu|$, then $\mu=\nu$ and $g\cdot x=x$ for all $x\in \SE^\infty$. It is now not hard to see that $g=a^{\ell \mathfrak N}$ for some $\ell\in \bZ$. 

If $|\mu|\ne |\nu|$, WLOG, $|\nu|> |\mu|$. Then $\nu=\mu\nu'$ for some $\nu'\in \SE^*\setminus\SE^0$ and $g\cdot x= \nu' x$. This is impossible as $n>1$. 
\end{proof}

Combining \cite[Theorem 6.6, Theorem 6.13]{LY21} with Corollary \ref{C:QO} yields 

\begin{thm}
\label{T:rk1Kir}
$\O_{\bZ, \mathsf{E}_n}$ satisfies UCT. It is simple iff $n_i\!\nmid\! m_i$ for some $1\le i\le \kappa$. So it is a Kirchberg algebra iff $n_i\!\nmid \! m_i$ for some $1\le i\le \kappa$. 
\end{thm}

\begin{rem}
\label{R:nmmn}
It is well-known that the roles of $n$ and $m$ in Baumslag-Solitar \textit{groups} $\BS(n,m)$ are symmetric in the sense of $\BS(n,m)\cong \BS(m,n)$. So $\ca(\BS(n,m))\cong \ca(\BS(m,n))$. 
However, the symmetry is lost for BS semigroups. For instance, If $0<n\ne m\in \bN$ satisfies $n\!\mid\!m$, then $\Q(\BS^+(n,m))$ is not simple while $\Q(\BS^+(m,n))$ is simple. 
\end{rem}

\subsection{Cartan subalgebras of $\O_{\bZ, \SE}$}
We begin with the definition of Cartan subalgebras. Let $\B$ be an abelian C*-subalgebra of a given C*-algebra $\A$. 
$\B$ is called a \textit{Cartan subalgebra} in $\A$ if 
\begin{itemize}
\item[(i)] $\B$ contains an approximate unit in $\A$;
\item[(ii)] $\B$ is a MASA;
\item[(iii)] $\B$ is regular: the normalizer set $N(\B)=\{x\in \A: x\B x^*\cup x^*\B x\subseteq \B\}$ generates $\A$;
\item[(iv)] there is a faithful conditional expectation $\E$ from $\A$ onto $\B$. 
\end{itemize}

In this subsection, we show that there is a canonical Cartan subalgebra for each $\O_{\bZ, \mathsf E_n}$. It is closely related to the fixed point algebra $\O_{\bZ, \SE}^\gamma$ of the 
gauge action $\gamma$. However, there is an essential difference between the cases of  $n=1$ and $n>1$.

\subsubsection{The case of $n=1$} 
The case of $n=1$ is a special case of Subsection \ref{S:ni=1} below with $\mathsf k=1$ (i.e., rank 1), which 
does not use any results from this section. So we record the result below just for completeness.

In what follows, to simplify our writing, let us set $\F:=\O_{\bZ, \mathsf E_n}^\gamma$, and $\F'$ to be the (relative) commutant of $\F$ in $\O_{\bZ, \mathsf E_n}$.

\begin{prop} 
\label{P:Cartan n=1}
Keep the above notation. Then 
$\F'=\ol\spn\{ s_{e^p} a^\ell s_{e^q}^*: m^p=m^q, \ell \in \bZ\}$ and $\F'$ is a Cartan subalgebra of $\O_{\bZ, \mathsf E_1}$. 
\end{prop}

It is worth noticing that there are three possible cases for $\F'$:
\begin{itemize}
\item
If $m=1$, then $\F'=\ol\spn\{ s_{e^p} u_{a^\ell} s_{e^q}^*: p,q\in \bN, \ell \in \bZ\}=\O_{\bZ, \mathsf E_1} \cong \rC(\bT^2)$. 

\item
If $m=-1$, then $\F'=\ol\spn\{ s_{e^p} u_{a^\ell} s_{e^q}^*: p,q\in \bN \text{ with }p-q\in 2\bZ, \ell \in \bZ\}$. 

\item 
If $m\ne \pm 1$, then $\F'=\F$. 
\end{itemize}


\subsubsection{The case of $n>1$}
Recall that $\mathfrak N =\lcm(n_i:1\le i\le \kappa)$ and $n=\sum\limits_{i=1}^\kappa n_i$. 
\begin{prop}
\label{P:Cartan n>1}
If $n>1$, then the cycline C*-subalgebra $\M:=\ca(s_\mu u_{a^{\ell\mathfrak{N}}} s_\mu^*: \mu \in \SE^*, \ell\in \bZ)$ is a MASA in $\O_{\bZ, \SE}$. 
\end{prop}

\begin{proof}
We first show that $\M$ is abelian. Compute 
\begin{align*}
&(s_\mu u_{a^{K\mathfrak{N}}}s_\mu^*)(s_\nu u_{a^{L\mathfrak{N}}}s_\nu^*)\\
&=
\begin{cases}
s_\mu u_{a^{(K+L)\mathfrak{N}}}s_\mu^*& \text{if }\mu=\nu,\\
s_\mu u_{a^{K\mathfrak{N}}}s_{\nu'} u_{a^{L\mathfrak{N}}}s_\nu^*=s_\mu s_{\nu'} u_{a^{K\mathfrak{N}}|_{\nu'}}u_{a^{L\mathfrak{N}}}s_\nu^*
=s_\nu u_{a^{K\mathfrak{N}}|_{\nu'}}u_{a^{L\mathfrak{N}}}s_\nu^* & \text{if }\nu=\mu\nu',\\
s_\mu u_{a^{K\mathfrak{N}}}s_{\mu'}^* u_{a^{L\mathfrak{N}}}s_\nu^*=s_\mu u_{a^{K\mathfrak{N}}} u_{(a^{-L\Bn}|_{a^{L\mathfrak{N}}\cdot\mu'})^{-1}}s_{\mu'}^* s_\nu^*
=s_\mu u_{a^{K\mathfrak{N}}} u_{a^{L\mathfrak{N}}|_{\mu'}}s_{\mu}^* & \text{if }\mu=\nu\mu',\\
0&\text{otherwise}.
\end{cases}
\end{align*}
Similar calculations yield
\begin{align*}
(s_\nu u_{a^{L\mathfrak{N}}}s_\nu^*)(s_\mu u_{a^{K\mathfrak{N}}}s_\mu^*)
=
\begin{cases}
 s_\nu u_{a^{(K+L)\mathfrak{N}}}s_\nu^* &\text{ if }\mu=\nu,\\
 s_\nu u_{a^{L\mathfrak{N}}} u_{a^{K\mathfrak{N}}|_{\nu'}}s_{\nu}^*& \text{ if }\nu=\mu\nu',\\
s_\mu u_{a^{L\mathfrak{N}}|_{\mu'}}u_{a^{K\mathfrak{N}}}s_\mu^*& \text{ if }\mu=\nu\mu',\\
0&\text{ otherwise}.
\end{cases}
\end{align*}
Thus $\M$ is abelian.

As in \cite{LY19$_2$}, let $\G_{\bZ, \SE}$ be the groupoid associated to the self-similar graph $(\bZ, \SE)$. It follows from \cite[Lemma 5.2]{LY19$_2$} and Proposition \ref{P:cycline} that $\text{Iso}(\G_{\bZ, \SE})^\circ=\bigcup\limits_{\mu\in \SE^*, \ell\in \bZ} Z(\mu, a^{\ell\mathfrak{N}}, \mu)$. 
Also notice that $\M\cong \ca(\text{Iso}(\G_{\bZ, \SE})^\circ)$. So $\ca(\text{Iso}(\G_{\bZ, \SE})^\circ)$ is abelian. 
Hence, by \cite[Corollary 5.4]{CRST21}, $\M$ is a MASA. 
\end{proof}

\begin{rem}
By Lemma \ref{L:per n=1}, when $n=1$, the cycline subalgebra of $\O_{\bZ, \mathsf{E}_1}$ coincides with $\O_{\bZ, \mathsf{E}_1}$, which is generally not abelian. So Proposition \ref{P:Cartan n>1} does not hold true for $n=1$. 
\end{rem}

\begin{rem}
Keep the same notation in the proof above. If $\text{Iso}(\G_{\bZ, \SE})^\circ$ is closed, then applying \cite[Corollary 4.5]{BNRSW16} one can conclude that $\M$ is Cartan in $\O_{\bZ, \SE}$. 
But, unfortunately, 
$\text{Iso}(\G_{\bZ, \SE})^\circ$ needn't be closed in general. 
\end{rem}
\subsection{Some old examples revisited}

Recall the flip (single-vertex) rank 2 graphs 
\[
\Lambda_{\text{flip}}=\langle\Be_i, \Bf_j: \Be_i \Bf_j = \Bf_i \Be_j, i, j\in [n]\rangle^+ \ (n\ge 2)
\]
and the square rank 2 graph 
\[
\Lambda_{\text{square}} = \langle \Bx_i, \By_j: \Bx_i \By_j=\By_{i+1} \Bx_j, i, j\in [2] \rangle^+. 
\]

\begin{eg}
\label{eg:m=n}
In \cite{CaHR16}, it is shown that $\Q(\BS^+(n,n))\cong \rC(\bT) \otimes \O_n$. From what we have obtained so far, we can prove this by relating to rank 2 graphs. In fact, we have 
\[
\Q(\BS^+(n,n))\cong \rC(\bT) \otimes \O_n\cong \O_{\Lambda_{\text{flip}}}.
\]
To see this, we construct an explicit isomorphism from  $\O_{\Lambda_{\text{flip}}}$ onto $\Q(\BS^+(n,n))$. Let $\pi:  \O_{\Lambda_{\text{flip}}}\to \Q(\BS^+(n,n))$ be the homomorphism determined by 
\begin{align}
\label{E:pi}
s_{\Be_i}\mapsto E_i:=s_{e_i}, \ s_{\Bf_j}\mapsto F_j:=u_{a^n} s_{e_j} \ (i, j\in [n]).
\end{align}
It is easy to see that $E_iF_j=F_iE_j$ as $u_{a^n}$ is in the center of $\Q(\BS^+(n,n))$. 
Also $\pi$ is surjective as from $a e_i=e_{i+1}$ $(0\le i\le n-2$) and $a e_{n-1}=e_0 a^n$ one has 
\begin{align*}
 \sum_{i=0}^{n-1}u_a s_{e_i}  s_{e_i}^*
=\sum_{i=0}^{n-2} s_{e_{i+1}} s_{ e_i}^* + s_{e_0} u_{a^n} s_{e_{n-1}}^*
=\sum_{i=0}^{n-2} s_{e_{i+1}} s_{e_i}^* + u_{a^n} s_{e_0} s_{e_{n-1}}^*
=\sum_{i=0}^{n-2} s_{e_{i+1}} s_{e_i}^* + F_0 s_{e_{n-1}}^*. 
\end{align*}
Conversely, define $\rho:  \Q(\BS^+(n,n))\to  \O_{\Lambda_{\text{flip}}}$ by 
\[
s_{e_i}\mapsto s_{\Be_i}, \ u_a\mapsto \sum_{i=0}^{n-2} s_{\Be_{i+1}}s_{\Be_i}^* + s_{\Bf_0} s_{\Be_{n-1}}^*. 
\]
Then $\rho$ determines a homomorphism. Also one can check that 
$\pi$ and $\rho$ are the inverse to each other. Therefore one has $ \Q(\BS^+(n,n))\cong  \O_{\Lambda_{\text{flip}}}$, 
which is also isomorphic to  $\rC(\bT) \otimes \O_n$ by \cite{DY091}.
\end{eg}

\begin{eg}
\label{eg:squareflip}
In this example, through $\BS^+(2,2)$, we are able to show that $\O_{\Lambda_{\text{square}}}\cong  \O_{\Lambda_{\text{flip}}}$, 
which seems unclear in \cite{DY091} although both 
$\O_{\Lambda_{\text{flip}}}$ and $\O_{\Lambda_{\text{square}}}$ are well-studied there.  

Let $W:=s_{\Be_1} s_{\Be_0}^*+ s_{\Bf_0} s_{\Be_1}^*$. Then $W^2=s_{\Bf_0} s_{\Be_0}^*+s_{\Bf_1} s_{\Be_1}^*$. Define 
$\pi: \O_{\Lambda_{\text{square}}}\to \O_{\Lambda_{\text{flip}}}$ via
\[
s_{\Bx_0}\mapsto s_{\Be_0},\ s_{\Bx_1}\mapsto s_{\Be_1} W^*,\ s_{\By_0}\mapsto W s_{\Be_0},\  s_{\By_1}\mapsto W s_{\Be_1} W^*.  
\]
Then one can verify that $\pi$ is a homomorphism.

Let $F:=s_{\By_1} s_{\Bx_1}^* + s_{\By_0} s_{\Bx_0}^*$. Then $F^2=\sum_{i, j\in [2]} s_{\By_i\By_j} s_{\Bx_{(i+1)}\Bx_j}^*$. 
Let $\rho: \O_{\Lambda_{\text{flip}}}\to \O_{\Lambda_{\text{square}}}$ be defined as 
\[
s_{\Be_0}\mapsto s_{\Bx_0},\ s_{\Be_1}\mapsto s_{\Bx_1} F,\ 
s_{\Bf_0}\mapsto  F^2 s_{\Bx_0}=s_{\Bx_0} F^2,\  s_{\Bf_1}\mapsto F^2 s_{\Bx_1} F=s_{\Bx_1} F^3 .  
\]
Then $\rho$ is a homomorphism. Moreover, $\pi$ and $\rho$ are the inverse to each other, and $\rho(W)=F$ and $\pi(F)=W$. 
\end{eg}

 
\section{Rank $\mathsf k$ case:  more than higher rank Baumslag-Solitar semigroups}
\label{S:rkBS}

In this section, we first propose a notion of higher rank BS semigroups $\Lambda_\theta(\mathfrak n, \mathfrak m)$. We then briefly describe how 
higher rank BS semigroups relate to Furstenberg’s $\times p, \times q$ conjecture. Our main focus here are two cases -- 
$\Lambda_{\mathsf d}(\mathfrak n, \mathds 1_{\mathsf k})$ and $\Lambda_{\mathsf d}(\mathds 1_{\mathsf k}, \mathfrak m)$. 
For $\Lambda_{\mathsf d}(\mathfrak n, \mathds 1_{\mathsf k})$, it is related to products of odometers studied in \cite{LY19$_1$}. Applying some results in \cite{LY19$_1$, LY19$_2$}, 
one can easily characterize the simplicity of $\O_{\Lambda_{\mathsf d}(\mathfrak n, \mathds 1_{\mathsf k})}$ 
and see that the cycline subalgebra is Cartan in $\O_{\Lambda_{\mathsf d}(\mathfrak n, \mathds 1_{\mathsf k})}$. But, here, we first show the fixed point algebra $\F$ of the gauge 
action $\gamma$ is a Bunce-Deddens algebra, and so $\F$ has a unique faithful tracial state $\tau$. Then composing with the conditional expectation $\Phi$
from $\O_{\Lambda_{\mathsf d}(\mathfrak n, \mathds 1_{\mathsf k})}$ onto $\F$ yields a state $\omega=\tau\circ\Phi$. We then study the associated von Neumann algebra
$\pi_\omega(\O_{\Lambda_{\mathsf d}(\mathfrak n, \mathds 1_{\mathsf k})})''$ in the same vein of \cite{Yan10, Yan12}. More precisely, we provide some characterizations 
of when $\pi_\omega(\O_{\Lambda_{\mathsf d}(\mathfrak n, \mathds 1_{\mathsf k})})''$ is a factor, and further obtain its type. 
For $\Lambda_{\mathsf d}(\mathds 1_{\mathsf k}, \mathfrak m)$, it is intimately related to Furstenberg’s $\times p, \times q$ conjecture. 
In this case, we obtain a canonical Cartan for $\O_{\Lambda_{\mathsf d}(\mathds 1_{\mathsf k}, \mathfrak m)}$, which is generally a proper 
subalgebra of its cycline subalgebra. We will continue studying $\O_{\Lambda_{\mathsf d}(\mathds 1_{\mathsf k}, \mathfrak m)}$ and its relative(s) in a forthcoming paper. 

\subsection{Higher rank Baumslag-Solitar semigroups}

Consider $\mathsf{k}$ given self-similar graphs $(\bZ, \mathsf{E}_{n_i})$ with $\mathsf{E}_{n_i}=\{\Bx_\fs^i: s\in [n_i]\}$ ($1\le i\le \mathsf k$). Suppose that $\Bx^i_\fs$'s  
satisfy the commutation relations $\theta_{ij}(\Bx_\fs^i, \Bx_\ft^j)= (\Bx_{\ft'}^j , \Bx_{\fs'}^i)$ for $1 \leq i < j \leq \mathsf{k}$. 
Applying \cite[Proposition 4.1]{LY19$_1$}, one has 

\begin{prop}
\label{P:prop4.1}
Keep the same notation. The $\mathsf k$ self-similar graphs $(\bZ, \mathsf{E}_{n_i})$ with the commutation relations $\theta_{ij}$'s 
determine a self-similar $\mathsf k$-graph $(\bZ, \Lambda_\theta)$ if and only if 
\begin{align}
\label{E:act}
 (a\cdot \Bx_\fs^i)(a|_{\Bx_\fs^i} \cdot \Bx_\ft^j) &= (a \cdot \Bx_{\ft'}^j)(a|_{\Bx_{\ft'}^j} \cdot \Bx_{\fs'}^i)
\end{align}
for all $1 \leq i < j \leq \mathsf{k}$, $\fs\in [n_i]$, $\ft\in [n_j]$. 
\end{prop}

Based on Example \ref{eg:BS}, it is reasonable to introduce the following notion. 

\begin{defn}
\label{D:rkBS}
A self-similar $\mathsf k$-graph obtained from $(n_i, m_i)$-odometers with the commutation relations $\theta_{ij}$'s is called a \textit{rank $\mathsf k$ BS semigroup}, denoted as 
$\Lambda_\theta((n_1, \ldots, n_{\mathsf k}), (m_1, \ldots, m_{\mathsf k}))$, or simply 
$\Lambda_\theta(\mathfrak n, \mathfrak m)$ if the context is clear. 
The ambient $\mathsf k$-graph is still written as $\Lambda_\theta$. 
\end{defn}

Here are some examples of higher rank BS semigroups. 

\begin{eg}
\label{eg:po}
A standard product of odometers studied in \cite{LY19$_1$} is a rank $\mathsf{k}$ BS semigroup 
induced from $(n_i, 1)$-odometers $\mathsf E(n_i, 1)$ ($1\le i\le \mathsf{k}$) with the division commutation relations (refer to Example \ref{eg:d} for $\mathsf d$). 
So it is of the form $\Lambda_{\mathsf d}(\mathfrak n, \mathds 1_{\mathsf k})$.
\end{eg}

\begin{eg} 
\label{eg:tri}
Consider $(n_i, m_i)$-odometers with the trivial permutation $\theta_{ij}$ ($1\le i<j\le \mathsf k$) (see Example \ref{eg:t}). Then they 
induce a rank $\mathsf k$ BS semigroup if and only if $n_i=1$ for each $1\le i\le \mathsf{k}$. In fact,
the condition \eqref{E:act} in Proposition \ref{P:prop4.1} implies $n_i = 1$ for all $1\le i\le \mathsf{k}$. 
Then the trivial relation is the same as the division commutation relation. 
The rank $\mathsf{k}$ BS semigroup obtained in this case is of the form $\Lambda_{\mathsf d}(\mathds 1_{\mathsf k}, \mathfrak m)$. 
\end{eg}

\begin{eg}
\label{eg:div}
Suppose that $n_i=n$ for all $1 \le i\le \mathsf{k}$, and that $\theta_{ij}$ is the division permutation. 
Then one necessarily has $\theta_{ij}(s,t)=(t,s)$. 
One can check that this yields a rank $\mathsf{k}$ BS semigroup if and only if either $m_i = m_j$ for all $1\le i,j\le \mathsf{k}$, or $n\!\mid\!m_i$ for all $1\le i\le \mathsf{k}$. 
These are rank $\mathsf{k}$ BS semigroups of the form 
$\Lambda_{\mathsf d}((n,\ldots, n), (m, \ldots, m))$, or
$\Lambda_{\mathsf d}((n,\ldots, n), (n\widetilde m_1, \ldots, n \widetilde m_{\mathsf k}))$.
 
So the class obtained in Example \ref{eg:tri} is a special case here with $n=1$. 
\end{eg}

In the sequel, we provide some interesting C*-algebras studied in the literature, which can be realized from higher rank BS semigroups. 
\begin{enumerate}
\item
It is shown in \cite{LY19$_1$} that the Cuntz algebra $\Q_\bN$ is isomorphic to the boundary quotient of a semigroup of the form in Example \ref{eg:po}. 

\item
For $2\le p\in \bN$, the $p$-adic  C*-algebra $\Q_p$ can also be recovered from a semigroup of the form in Example \ref{eg:po} 
(cf. \cite{LY19$_1$} for $p=2$ in terms of standard product of odometers).

\item 
$\Q_\bN$ can be also realized as a boundary quotient of the left $ax+b$ semigroup $\bN\rtimes \bN^\times$ studied in \cite{LR10}. The C*-algebra of $\bN\rtimes \bN^\times$ itself is also related 
to higher rank BS semigroups. 

\item
The boundary $\partial \T(\bN^\times\ltimes \bN)$ of the right $ax+b$ semigroup $\bN^\times\ltimes \bN$ is a boundary quotient of a higher-rank BS semigroup of the form given in Example \ref{eg:tri}. 

\item
The C*-algebra $\O(E_{n,m})$ studied by Katsura in \cite{Kat08} is isomorphic to $\O_{\mathsf E(n,m)}$ (see Examples \ref{eg:E(n,m)} and \ref{eg:BS}). 
\end{enumerate}

\begin{rem}
The C*-algebras of both the left $ax+b$ semigroup $\bN\rtimes \bN^\times$  and the right $ax+b$ semigroup $\bN^\times\ltimes \bN$ are related to higher rank BS semigroups. This will be 
studied elsewhere. 
\end{rem}

\subsection{Relation to Furstenberg's $\times p, \times q$ conjecture}
\label{SS:Fur}
We first give a very brief introduction on Furstenberg's $\times p, \times q$ conjecture. For all undefined notions or any further information, refer to \cite{BS24, Fur67, HW17, Sca20}. 

Let $2\le p,q\in \bN$ be multiplicatively independent, i.e., $\frac{\ln p}{\ln q}\not\in \bQ$. Define $T_p: \bT\to \bT$ by $T_p(z)=z^p$ for all $z\in \bT$. Similarly for $T_q$. 
A subset of $\bT$ is said to be \textit{$\times p, \times q$-invariant} if it is invariant under both $T_p$ and $T_q$. 
Furstenberg classifies all closed $\times p, \times q$-invariant subsets of $\bT$ in \cite{Fur67}: Such a subset is either finite or $\bT$ itself. 
Then he conjectures the following: 

\begin{conjecture*}[Furstenberg's $\times p, \times q$ conjecture]
An ergodic $\times p, \times q$-invariant Borel probability measure of $\bT$ is either finitely supported or the Lebesgue measure. 
\end{conjecture*}

According to our best knowledge, this conjecture is still open. 
The best known result so far is the following theorem, which is proved by Rudolph when $p$ and $q$ are coprime in \cite{Rud90}, and later improved by Johnson in \cite{Jon92}. 

\begin{theorem*}[Rudolph-Johnson]
If $\mu$ is an ergodic $\times p, \times q$-invariant measure on $\bT$, then either both entropies of $T_p$ and $T_q$ with respect to $\mu$ are $0$, or $\mu$ is the Lebesgue measure.
\end{theorem*}

When both entropies of $T_p$ and $T_q$ with respect to $\mu$ are $0$, Furstenberg's $\times p, \times q$ conjecture is reduced to studying the 
C*-algebra $\ca(G)$ of the group $G$, where 
\begin{align}
\label{E:p1q10}
G:=\langle s, t, z: st=ts, \ sz=z^p s,\ tz=z^q t\rangle. 
\end{align}
It turns out that $\ca(G)\cong \ca(\bZ[\frac{1}{pq}])\rtimes \bZ^2\cong \ca(\bZ[\frac{1}{pq}]\rtimes \bZ^2)$. In \cite{HW17}, the representation theory of $\ca(\bZ[\frac{1}{pq}])\rtimes \bZ^2$ is 
studied. In particular, the authors focus on which kind of its representations are induced by $\times p, \times q$-invariant measures on $\bT$.
Later, the following equivalence is shown in \cite{BS24}: 
\textit{Furstenberg's $\times p, \times q$ conjecture holds true if and only if the canonical trace is the only faithful extreme tracial state on $\ca(G)\cong \ca(\bZ[\frac{1}{pq}]\rtimes \bZ^2)$.}

Our purpose here is to connect Furstenberg's conjecture with higher rank BS semigroups. 
For this, from \eqref{E:p1q10} one has 
\begin{align*}
t^{-1}s^{-1}=s^{-1}t^{-1}, \ z^{-1}s^{-1}=s^{-1}z^{-p} ,\ z^{-1}t^{-1}=t^{-1}z^{-q}. 
\end{align*}
Let $\tilde G$ be the group
\begin{align}
\label{E:1p1q0}
\tilde G:=\langle \tilde s, \tilde t, \tilde z: \tilde t\tilde s=\tilde s\tilde t, \ \tilde z\tilde s=\tilde s \tilde z^{p} ,\ \tilde z \tilde t=\tilde t \tilde z^{q}\rangle. 
\end{align}
Thus  $G\cong \tilde G$ and so $\ca(G)\cong \ca(\tilde G)$. 
The upshot by doing so is $\ca(G)\cong \O_{\Lambda_{\mathsf d}((1,1),(p,q))}$. 

Now return to \eqref{E:p1q10} again. Let $G^+$ be the corresponding semigroup
\begin{align}
\label{E:p1q10+}
G^+:=\langle s, t, z: st=ts, \ sz=z^p s,\ tz=z^q t\rangle^+. 
\end{align}
Then we claim that $\Q(G^+)\cong \O_{\Lambda_{\mathsf d}((p,q),(1,1))}$.  
In fact, let $e_i:=z^i s$ and $f_j:=z^j t$ for $i\in [p]$ and $j\in [q]$. Then 
$st=ts\iff e_0 f_0=f_0e_0$. For $k\in [p]$ and $\ell \in [q]$, let $k'\in [p]$ and $\ell'\in [q]$ be the unique ones such that 
$k+\ell p= \ell' + k' q$. Then we have 
\begin{align*}
e_k f_\ell =z^k s z^\ell t=z^k z^{\ell p} st= z^{\ell'} z^{k' q} ts=z^{\ell'} tz^{k'}s=f_{\ell'} e_{k'}. 
\end{align*}
Thus there is a homomorphism 
\[
\pi:  \O_{\Lambda_{\mathsf d}((p,q),(1,1))} = \O_{\bZ, \Lambda_{\mathsf d}} \to \Q(G^+),
\ s_{\Bx_i^1}\mapsto v_{e_i},\ s_{\Bx_j^2}\mapsto v_{f_j},\ u_a\mapsto v_z,
\] 
which is an isomorphism as it has an inverse given by 
\[
v_s\mapsto s_{\Bx^1_0},\ v_t\mapsto s_{\Bx^2_0},\ v_z\mapsto u_a. 
\]
Hence $\O_{\Lambda_{\mathsf d}((p,q),(1,1))} \cong \Q(G^+)$. 

To sum up, we have shown that
\begin{align}
\label{E:C*(G)Q(G+)}
\ca(G)\cong  \O_{\Lambda_{\mathsf d}((1,1),(p,q))} \text{ and }\Q(G^+)\cong  \O_{\Lambda_{\mathsf d}((p,q),(1,1))}. 
\end{align}

Based on the above, in what follows, we focus on two extreme, but rather interesting, classes of higher rank BS semigroups: 
$\Lambda_{\mathsf d}(\mathfrak n, \mathds 1_{\mathsf k})$ with $\mathds 1_{\mathsf k}\le \mathfrak n\in \bN^{\mathsf k}$ and $\Lambda_{\mathsf d}(\mathds 1_{\mathsf k}, \mathfrak m)$
with $m_i\ne 0$ for all $1\le i\le \mathsf k$.

\smallskip

From now on, to unify our notation,  the set  $\{0,1\ne p_i \in \bZ: 1\le i\le \mathsf k\}$ is said to be \textit{multiplicatively independent} if there is no $\mathbf 0\ne q\in \bZ^{\mathsf k}$ such that 
$\prod\limits_{i=1}^{\mathsf k} p_i^{q_i}=1$. When all $p_i$'s are also $\ge 1$, $\{p_i: 1\le i\le \mathsf k\}$ is multiplicatively independent if and only if $\{\ln p_i: 1\le i\le \mathsf k\}$ is 
rationally independent.

\subsection{The case of $\mathfrak m=\mathds 1_{\mathsf k}$}
\label{S:type}

The C*-algebra of the self-similar $\mathsf k$-graph  $\Lambda_{\mathsf d} (\mathfrak n, \mathds 1_{\mathsf k})$ is studied in \cite{LY19$_1$}. 
It is shown there that 
$\O_{\Lambda_{\mathsf d}(\mathfrak n, \mathds 1_{\mathsf k})}$ is simple if and only if $\{n_i: 1\le i\le \mathsf k\}$ is multiplicatively independent, 
and that its cycline subalgebra is Cartan by \cite[Theorem 5.6]{LY19$_2$}.
In what follows, we identify the center of $\O_{\Lambda_{\mathsf d}(\mathfrak n, \mathds 1_{\mathsf k})}$, which is overlooked in \cite{LY19$_1$, LY19$_2$}. 
A useful lemma first: 

\begin{lem}
\label{L:levtr}
(i) Let $\mu,\nu\in \Lambda_{\mathsf d}^{\mathbf p}$ $(\mathsf p\in \bN^{\mathsf k}$) and $m\in \bZ$. Then there is $\ell\in \bZ$ 
such that $ \nu a^m = a^\ell \mu$. 

(ii) $\O_{\Lambda_{\mathsf d}(\mathfrak n, \mathds 1_{\mathsf k})}=\ol\spn\{u_{a^m} s_\alpha s_\beta^*: m\in \bZ, \alpha, \beta\in \Lambda_{\mathsf d}\}$.
\end{lem}

\begin{proof}
(i) We first prove the lemma holds true for the classical odometer action $\mathsf E(n, 1)$. We argue this by induction with respect to the lengths of $\mu$ and $\nu$. 
If $\mu:=e_i$ and $\nu:=e_j$. WLOG we assume that $i\ge j\in [n]$.  For any $m \in \bZ$, by Lemma \ref{L:formula} one can verify that  $a^{mn+i-j} \nu = \mu a^m$.  
Assume that this is true for any $m\in \bZ$ and all $\mu, \nu \in \SE^*$ with $|\mu|=|\nu|\le k$. Now consider $e_s \mu$ and $e_t \nu$  with $|\mu|=|\nu|=k$. Let $m\in \bZ$. 
By our inductive assumption, we have $e_s \mu a^m = e_s a^{\ell'} \nu = a^{\ell} e_t\nu$ for some $\ell', \ell\in \bZ$. This proves the $\mathsf E(n, 1)$ case. 

Now return to $\Lambda_{\mathsf d}$. Let $\mu, \nu\in \Lambda_d$ with $d(\mu)=d(\nu)=\mathbf p\in \bN^{\mathsf k}$. Then by the unique factorization property 
$\mu = \mu_1\cdots \mu_{\mathsf k}$ and $\nu = \nu_1\cdots \nu_{\mathsf k}$ with 
$\mu_i, \nu_i \in \Lambda_{\mathsf d}^{p_i \epsilon_i}$ for $1\le i\le \mathsf k$. 
Then for any $m\in \bZ$, apply the above to $\mu_{\mathsf k} a^m$ in $\mathsf E(n_{\mathsf k}, 1)$, there is $\ell_{\mathsf k}\in \bZ$ such that 
$\mu a^m = \mu_1\cdots \mu_{\mathsf k -1} a^{\ell_{\mathsf k}} \nu_{\mathsf k}$. 
Repeatedly using the above gives $\mu a^m = a^\ell \nu$ for some $\ell \in \bZ$. 

(ii) This follows from (i) and Proposition \ref{P:genO} (i). 
\end{proof}

By  \cite[Theorem 7.5]{LY19$_2$} (or Appendix there), $\Lambda_{\mathsf d}(\mathfrak n, \mathds 1_{\mathsf k})$ and $\Lambda_{\mathsf d}$ share the same periodicity 
$\{\mathbf p\in \bZ^{\mathsf k}: \mathfrak n^{\mathbf p}=1\}$. Thus, for each pair $(\mathbf p, \mathbf q)\in \bN^{\mathsf k}\times \bN^{\mathsf k}$ with $\mathfrak n^{\mathbf p}=\mathfrak n^{\mathbf q}$, there is a bijection 
$\phi_{\mathbf p, \mathbf q}: \Lambda_{\mathsf d}^{\mathbf p} \to \Lambda_{\mathsf d}^{\mathbf q} $ satisfying 
\begin{align}
\label{E:uv}
\mu \nu = \phi_{\mathbf p, \mathbf q}(\mu) \phi_{\mathbf p, \mathbf q}^{-1}(\nu),\ \phi_{\mathbf p, \mathbf q}^{-1}(\nu)\phi_{\mathbf p, \mathbf q}(\mu)=\nu \mu
\end{align} 
for every pair $(\mu, \nu)\in \Lambda_{\mathsf d}^{\mathbf p}\times \Lambda_{\mathsf d}^{\mathbf q}$ (\cite[Theorem 7.1]{DY09}) (or \cite[Section 5]{HLRS15}). 
Let 
\[
V_{\mathbf p, \mathbf q}:=\sum_{\mu\in \Lambda_{\mathsf d}^{\mathbf p}} s_\mu s_{\phi_{\mathbf p, \mathbf q}(\mu)}^*. 
\]
By (\cite[Theorem 4.9]{DY09}), each $V_{\mathbf p, \mathbf q}$ is a unitary in $\O_{\Lambda_{\mathsf d}(\mathfrak n, \mathds 1_{\mathsf k})}$.

\begin{prop}
\label{P:center}
The center of $\O_{\Lambda_{\mathsf d}(\mathfrak n, \mathds 1_{\mathsf k})}$ is given by 
\[
\Z(\O_{\Lambda_{\mathsf d}(\mathfrak n, \mathds 1_{\mathsf k})}) =\ca(V_{\mathbf p, \mathbf q}: (\mathbf p, \mathbf q)\in \bN^{\mathsf k}\times \bN^{\mathsf k},\ \mathfrak n^{\mathbf p}=\mathfrak n^{\mathbf q}).
\] 
In particular,  $\Z(\O_{\Lambda_{\mathsf d}(\mathfrak n, \mathds 1_{\mathsf k})})$ is trivial, if and only if $\{n_i: 1\le i\le \mathsf k\}$ is multiplicatively independent. 
\end{prop} 

\begin{proof}
Suppose that $A\in \Z(\O_{\Lambda_{\mathsf d}(\mathfrak n, \mathds 1_{\mathsf k})})$. Then by Proposition \ref{P:genO}, one has $A\in \Z(\O_{\Lambda_{\mathsf d}})$. 
But 
$\Z(\O_{\Lambda_{\mathsf d}})=\ca(V_{\mathbf p, \mathbf q}: (\mathbf p, \mathbf q)\in \bN^{\mathsf k}\times \bN^{\mathsf k},\ \mathfrak n^{\mathbf p}=\mathfrak n^{\mathbf q})$ (\cite{DY09}). 

It remains to show that each $V_{\mathbf p, \mathbf q}$ is indeed in $\Z(\O_{\Lambda_{\mathsf d}(\mathfrak n, \mathds 1_{\mathsf k})})$. 
This has been proved in \cite[Theorem 4.9]{DY09}. In what follows, we prove this by invoking Lemma \ref{L:levtr}. 

We first show that $u_a$ commutes with $V_{\mathbf p, \mathbf q}$. 
Then on one hand one has 
\begin{align*}
a\cdot (\mu \nu)=a\cdot(\phi_{\mathbf p, \mathbf q}(\mu) \phi_{\mathbf p, \mathbf q}^{-1}(\nu))
\implies a\cdot \mu a|_\mu\cdot \nu = a\cdot \phi_{\mathbf p, \mathbf q}(\mu) a|_{\phi(\mu)}\cdot \phi_{\mathbf p, \mathbf q}^{-1}(\nu).
\end{align*}
On the other hand, we have 
\[
a\cdot \mu a|_\mu\cdot \nu = \phi_{\mathbf p, \mathbf q}(a\cdot \mu) \phi_{\mathbf p, \mathbf q}^{-1}(a|_\mu\cdot \nu). 
\]
Thus 
\begin{align}
\nonumber
& a^{-1}\cdot \phi_{\mathbf p, \mathbf q}(\mu) a^{-1}|_{\phi_{\mathbf p, \mathbf q}(\mu)}\cdot \phi_{\mathbf p, \mathbf q}^{-1}(\nu)
=\phi_{\mathbf p, \mathbf q}(a^{-1}\cdot \mu) \phi_{\mathbf p, \mathbf q}^{-1}(a^{-1}|_\mu\cdot \nu)\\
\label{E:equi}
\implies & a^{-1}\cdot \phi_{\mathbf p, \mathbf q}(\mu) 
=\phi_{\mathbf p, \mathbf q}(a^{-1}\cdot \mu) \quad \text{and} \quad a^{-1}|_{\phi_{\mathbf p, \mathbf q}(\mu)}\cdot \phi_{\mathbf p, \mathbf q}^{-1}(\nu)=\phi_{\mathbf p, \mathbf q}^{-1}(a^{-1}|_\mu\cdot \nu)
\end{align}
Then one can easily see $g\cdot \phi_{\mathbf p, \mathbf q}(\mu) = \phi_{\mathbf p, \mathbf q}(g\cdot \mu)$ for any $g\in \bZ$. 
Similarly, one gets from the 2nd identity of \eqref{E:uv} that $g\cdot \phi_{\mathbf p, \mathbf q}^{-1}(\nu) = \phi_{\mathbf p, \mathbf q}^{-1}(g\cdot \nu)$ for any $g\in \bZ$. 
Thus the 2nd identity of \eqref{E:equi} yields $a^{-1}|_{\phi_{\mathbf p, \mathbf q}(\mu)}=a^{-1}|_\mu$, and so 
$
(a^{-1}|_{\phi_{\mathbf p, \mathbf q}(\mu)})^{-1}=(a^{-1}|_\mu)^{-1}.
$
Therefore 
\begin{align}
\label{E:equi2}
a|_{a^{-1}\cdot \phi_{\mathbf p, \mathbf q}(\mu)}=a|_{a^{-1}\cdot \mu}.
\end{align}

Now compute 
\begin{align*}
u_a V_{\mathbf p, \mathbf q}&
= \sum_{d(\mu)=\mathbf p} s_{a\cdot \mu}u_{a|_\mu} s_{\phi_{\mathbf p, \mathbf q}(\mu)}^*
= \sum_{d(\mu)=\mathbf p} s_\mu u_{a|_{a^{-1}\cdot \mu}} s_{\phi_{\mathbf p, \mathbf q}(a^{-1}\cdot\mu)}^*\\
&= \sum_{d(\mu)=\mathbf p} s_\mu u_{a|_{a^{-1}\cdot \phi_{\mathbf p, \mathbf q}(\mu)}} s_{a^{-1}\cdot\phi_{\mathbf p, \mathbf q}(\mu)}^* \ 
(\text{from } \eqref{E:equi}\text{ and }\eqref{E:equi2})\\
&= \sum_{d(\mu)=\mathbf p} s_\mu u_{({a^{-1}|_ {\phi_{\mathbf p, \mathbf q}(\mu)})^{-1}}} s_{a^{-1}\cdot\phi_{\mathbf p, \mathbf q}(\mu)}^*
= \left(\sum_{d(\mu)=\mathbf p} s_\mu s_{\phi_{\mathbf p, \mathbf q}(\mu)}^*\right)u_a\\
&= V_{\mathbf p, \mathbf q}\, u_a . 
\end{align*}
Then 
\begin{align*}
(u_{a^m} s_\alpha s_\beta^*) V_{\mathbf p, \mathbf q}
= u_{a^m} V_{\mathbf p, \mathbf q} s_\alpha s_\beta^*
= V_{\mathbf p, \mathbf q} (u_{a^m} s_\alpha s_\beta^*).
\end{align*}
By Lemma \ref{L:levtr} (ii), 
$\sum\limits_{d(\mu)=\mathbf p} s_\mu s_{\phi_{\mathbf p, \mathbf q}(\mu)}^*\in \Z(\O_{\Lambda_{\mathsf d}(\mathfrak n, \mathds 1_{\mathsf k})})$. 
\end{proof}

\bigskip

For $\Bm\in \bN^{\mathsf{k}}$, 
let $\F_{\Bm}:=\ol\spn\{s_\mu a^n s_\nu^*: \mu,\nu\in \Lambda_{(\mathfrak n, \mathds 1_{\mathsf k})} \text{ with } d(\mu)=d(\nu)=\Bm, n\in \bZ\}$, 
and so $\F=\ol{\bigcup\limits_{m\in \bN^{\mathsf{k}}} \F_{\Bm}}^{\|\cdot\|}$. 
Also notice that, due to the Cuntz-Krieger relations, we have $\F=\lim\limits_{\stackrel{\longrightarrow}{m\in \bN}}\F_{m\mathds{1}_{\mathsf k}}$. 

Let $d:=\prod\limits_{i=1}^{\mathsf{k}}\prod\limits_{p\in \P, \, p|n_i} p$, the product of all primes dividing some $n_i$'s. 

\begin{lem}
\label{L:BD}
$\F$ is a Bunce-Deddens algebra of type $d^\infty$. In particular, $\F$ has a unique faithful tracial state. 
\end{lem}

\begin{proof}
The proof below is motivated by the proof of \cite[Theorem 3.16]{LY212}.

For $\Bm \in \bN^{\mathsf k}$, let $\{e_{\mu,\nu}\}_{\mu,\nu\in \Lambda^{\Bm}_{\mathsf d}}$ be the matrix entries of $K(\ell^2(\Lambda_{\mathsf d}^\Bm)).$\footnote{In our case, 
$K(\ell^2(\Lambda_{\mathsf d}^{\Bm}))\cong M_{\mathfrak n^{\Bm}}(\bC)$.} 
To simplify our notation, let $\{a^n\}_{n\in \bZ}$ be the generating unitaries of $\ca(\bZ)$. Clearly, there is a homomorphism
\[
\rho_1: K(\ell^2(\Lambda_{\mathsf d}^{m\mathds{1}_{\mathsf k}}))\to \F_{m\mathds{1}_{\mathsf k}},\ e_{\mu,\nu}\mapsto s_\mu s_\nu^*,
\]
and 
\[
\rho_2:\rC(\bT)\to \F_{m\mathds{1}_{\mathsf k}},\ a^n\mapsto \sum_{\mu\in\Lambda_{\mathsf d}^{m\mathds{1}_{\mathsf k}}} s_\mu u_{a^n} s_\mu^*. 
\]
Some simple calculations show the images of $\rho_1$ and $\rho_2$ commute. By \cite[Theorem 6.3.7]{Mur04}, there is a homomorphism 
$\rho:  K(\ell^2(\Lambda_{\mathsf d}^{m\mathds{1}_{\mathsf k}}))\otimes \ca(\bZ)\to \F_{m\mathds{1}_{\mathsf k}}$\footnote{The tensor product $\otimes$ here could be either $\otimes_{\text{max}}$ or $\otimes_{\text{min}}$
as both factors are amenable.} satisfying $\rho(e_{\mu, \nu}\otimes a^n)= s_\mu u_{a^n} s_\nu^*$. It is not hard to see that $\rho$ is also invertible and so an isomorphism.

Set $\mu_0:=\Bx_0^1\cdots \Bx^{\mathsf k}_0$ and $\mu_{\mathfrak n -\mathds{1}}:=\Bx^1_{n_1-1}\cdots \Bx^{\mathsf k}_{n_{\mathsf k}-1}$. 
Embed $K(\ell^2(\Lambda_{\mathsf d}^{m\mathds 1_{\mathsf k}}))\otimes \ca(\bZ)$ into $K(\ell^2(\Lambda_{\mathsf d}^{(m+1)\mathds 1_{\mathsf k}}))\otimes \ca(\bZ)$ as follows:
\[
e_{\mu, \nu}\otimes a^n\mapsto \sum\limits_{\alpha\in \Lambda_{\mathsf d}^{\mathds{1}}}e_{\mu a^n\cdot \alpha, \nu\alpha}\otimes a^n|_\alpha. 
\]
Notice that 
$a|_\alpha=a$ if $\alpha=\mu_{\mathfrak n -\mathds{1}_{\mathsf k}}$, 
and $a|_\alpha=0$, otherwise. 
 
Now we have 
 \begin{align*}
 u_a
&=u_a \sum_{\mu\in\Lambda^{m\mathds{1}_{\mathsf k}}} s_\mu s_\mu^*
   =\sum_{\mu\in\Lambda_{\mathsf d}^{m\mathds{1}_{\mathsf k}}} s_{a\cdot\mu} u_{a|_\mu} s_\mu^*\\
&=\sum_{\mu\ne \mu_{\mathfrak n-\mathds{1}_{\mathsf k}}} s_{a\cdot\mu} s_\mu^*+s_{a\cdot\mu_{\mathfrak n-\mathds{1}_{\mathsf k}}} u_a s_{\mu_{\mathfrak n-\mathds{1}_{\mathsf k}}}^*\\
&=\sum_{\mu\ne \mu_{\mathfrak n-\mathds{1}_{\mathsf k}}} s_{a\cdot\mu} s_\mu^*+s_{\mu_0} u_a s_{\mu_{\mathfrak n-\mathds{1}_{\mathsf k}}}^*. 
\end{align*}

Therefore $\F$ is isomorphic to a Bunce-Deddens algebra of type of $d^\infty$, and so it has a unique faithful tracial state (\cite{Dav96}). 
\end{proof}

\begin{rem}
It is worth mentioning the following: If $S:=\{n_1,\ldots, n_{\mathsf k}\}\subset \bN$ is a set of mutually coprime natural numbers, then $\F$ is isomorphic to $B_S$ in \cite{BOS18}.  
\end{rem}


By Lemma \ref{L:BD} and \cite{Dav96}, 
$\F$ has a unique faithful tracial state, say $\tau$, given by 
\[
\tau(s_\mu u_{a^n} s_\mu^*)=\mathfrak n^{-d(\mu)}\delta_{n,0}.
\]
Recall that $\Phi_{\mathbf 0}$ is the faithful conditional expectation from $\O_{\Lambda_{\mathsf d}(\mathfrak n, \mathds 1)}$ onto $\F$ via the gauge action $\gamma$
(see Subsection \ref{S:sssC*}). 
Then $\omega:=\tau\circ \Phi_{\mathbf 0}$ is a state of $\O_{\Lambda_{\mathsf d}(\mathfrak n, \mathds 1_{\mathsf k})}$. 
 

We recall the notion of KMS states from \cite{BR97} (also see \cite{CM08}), and give some basic properties of KMS states for 
$\O_{\Lambda_{\mathsf d}(\mathfrak n, \mathds 1_{\mathsf k})}$.

\begin{defn}
Let $A$ be a C*-algebra, $\alpha$ be an action of $\bR$ on $A$,
and $A^a$ be the set of all analytic elements of $A$. Let  $0<\beta<\infty$. A state $\tau$ of $A$ is called a \emph{KMS$_\beta$ state} of $(A,\mathbb{R},\alpha)$ if $\tau(xy)=\tau(y \alpha_{i\beta}(x))$ for all $x,y \in A^a$.
\end{defn}

Recall the gauge action $\gamma:\bT^{\mathsf k} \to \Aut(\O_{\Lambda_{\mathsf d}(\mathfrak n, \mathds 1_{\mathsf k})})$:
\[
\gamma_z(s_\mu)=z^{d(\mu)}s_\mu\text{ and } \gamma_z(u_g)=u_g\qforal  z \in \bT^{\mathsf k}, \mu \in \Lambda,g \in \bZ.
\]
Let $\mathsf{r}=(\ln n_1, \ldots, \ln n_{\mathsf k})\in \bR^{\mathsf k}$. Define a strongly continuous homomorphism 
$\alpha^{\mathsf{r}}:\mathbb{R} \to \Aut(\O_{\Lambda_\mathsf d(\mathfrak n,\mathds 1_{\mathsf k})})$ by 
$\alpha^{\mathsf{r}}_t:=\gamma_{e^{it\mathsf{r}}}$. Notice that, for $\mu,\nu \in \Lambda_\mathsf d(\mathfrak n, \mathds 1_{\mathsf k}), g \in \bZ$, 
the function $\mathbb{C} \to \O_{\Lambda_\mathsf d(\mathfrak n, \mathds 1)}$, 
$\xi \mapsto e^{i \xi \mathsf{r} \cdot (d(\mu)-d(\nu))}s_\mu u_g s_\nu^*$ is an entire function.
So $s_\mu u_g s_\nu^*$ is an analytic element. By Proposition~\ref{P:genO}, in order to check the KMS$_\beta$ condition, it is sufficient to check whether it is valid on the set 
$\{s_\mu u_g s_\nu^*:\mu,\nu \in \Lambda_\mathsf d(\mathfrak n, \mathds 1_{\mathsf k}), g\in \bZ\}$.
In this subsection, we study basic properties of KMS$_\beta$ states of the one-parameter dynamical system $(\mathcal{O}_{\Lambda_\mathsf d(\mathfrak n, \mathds 1)},\mathbb{R},\alpha^{\mathsf{r}})$.

\begin{lem}
\label{L:wKMS}
Suppose that $\{n_i: 1\le i\le \mathsf k\}$ is multiplicatively independent. Then 
$\omega$ is the unique KMS$_1$ state on $\O_{\Lambda_\mathsf d(\mathfrak n, \mathds 1_{\mathsf k})}$. 
\end{lem}

\begin{proof}
An easy calculation shows 
\begin{align*}
\omega(s_\mu u_{a^n} s_\nu^*)=\delta_{\mu, \nu} \delta_{n,0} \, \mathfrak n^{-d(\mu)}. 
\end{align*}
By \cite[Theorems 6.11 and 6.12]{LY19$_2$}, $\omega$ is the unique KMS$_1$ state of $\O_{\Lambda_\mathsf d(\mathfrak n, \mathds 1)}$ (also see \cite[Theorem 7.1]{LY19$_2$}). 
\end{proof}

As in \cite{Yan10, Yan12}, let $L^2(\O_{\Lambda_\mathsf d(\mathfrak n, \mathds 1_{\mathsf k})})$ be the GNS Hilbert space determined by the state $\omega$: 
$\langle A, B\rangle:=\omega(A^*B)$ for all $A, B\in \O_{\Lambda_\mathsf d(\mathfrak n, \mathds 1_{\mathsf k})}$. 
For $A \in  \O_{\Lambda_\mathsf d(\mathfrak n, \mathds 1_{\mathsf k})}$, we denote the left action of $A$ by $\pi (A): \pi(A)B = AB$ for all 
$B \in  \O_{\Lambda_\mathsf d(\mathfrak n, \mathds 1)}$. Let $ \O_{\Lambda_\mathsf d(\mathfrak n, \mathds 1_{\mathsf k})}^c$ stand for the algebra as the finite linear 
span of the generators $s_\mu u_g s_\nu^*$. 

Define
\begin{align*}
S(A):=A^*,\ F(s_\mu u_{a^n} s_\nu^*):=\mathfrak n^{d(\mu)-d(\nu)} s_\nu u_{a^{-n}} s_\mu^*, \qfor A\in  \O_{\Lambda_\mathsf d(\mathfrak n, \mathds 1_{\mathsf k})}^c. 
\end{align*}
Then $F=S^*$. Also, if 
\[
J(s_\mu u_{a^n} s_\nu^*):=\mathfrak n^{\frac{d(\mu)-d(\nu)}{2}} s_\nu u_{a^{-n}} s_\mu^*, \ 
\Delta (s_\mu u_{a^n} s_\nu^*):=\mathfrak n^{d(\nu)-d(\mu)} s_\mu u_{a^n} s_\nu^*,
\]
one has 
\[
S=J\Delta^{\frac{1}{2}}=\Delta^{-\frac{1}{2}} J,\ F=J\Delta^{-\frac{1}{2}}= \Delta^{\frac{1}{2}}J. 
\]

Let $\pi_\omega(\O_{\Lambda_d(\mathfrak n, \mathds 1_{\mathsf k})})''$ be the von Neumann algebra generated by the GNS representation of $\omega$. 
Then $\pi_\omega(\O_{\Lambda_d(\mathfrak n, \mathds 1_{\mathsf k})})''$ coincides with the left von Neumann algebra of 
$ \O_{\Lambda_\mathsf d(\mathfrak n, \mathds 1_{\mathsf k})}^c$.

The proof of the following theorem can now be easily adapted from \cite{Yan10, Yan12} combined with \cite{LY19$_1$, LY19$_2$}, and is left to the interested reader. 

\begin{thm}
\label{T:factor}
The following statements are equivalent:
\begin{itemize}
\item[(i)] $\O_{\Lambda_\mathsf d(\mathfrak n, \mathds 1_{\mathsf k})}$ is simple.
\item[(ii)] $\O_{\Lambda_{\mathsf d}}$ is simple.
\item[(iii)] $\{n_i: 1\le i\le \mathsf k\}$ is multiplicatively independent.
\item[(iv)] $\Lambda_{\mathsf d}(\mathfrak n, \mathds 1_{\mathsf k})$ is aperiodic. 
\item[(v)] The ambient $\mathsf k$-graph $\Lambda_{\mathsf d}$ is aperiodic. 
\item[(vi)] $\pi_\omega(\O_{\Lambda_\mathsf d(\mathfrak n, \mathds 1_{\mathsf k})})''$ is a factor.
\end{itemize}

When $\pi_\omega(\O_{\Lambda_\mathsf d(\mathfrak n, \mathds 1_{\mathsf k})})''$ is a factor, 
\begin{itemize}
\item
it  is an AFD factor of type III$_{\frac{1}{n}}$ if $\mathsf k=1$; and 
\item
it is an AFD factor of type III$_1$ if $\mathsf k\ge 2$.
\end{itemize}
\end{thm}

\subsection{The case of $\mathfrak n=\mathds 1_{\mathsf k}$} 
\label{S:ni=1}

Since there is only one edge for each color $i$, in order to ease our notation, we write $\Bx_i$ (instead of the notation $\Bx_1^i$ used above) for this unique edge.
Thus $\Bx_i \Bx_j=\Bx_j \Bx_i$ for all $1\le i\ne j\le \mathsf k$. 
Since this is the unique commutation relation on $\Bx_i$'s, we denote $\Lambda_\mathsf d(\mathds 1_{\mathsf k}, \mathfrak m)$ simply as 
$\Lambda(\mathds 1_{\mathsf k}, \mathfrak m)$. The ambient $\mathsf k$-graph is just denoted as $\Lambda_{\mathds 1_{\mathsf k}}$. 

We should mention that it seems that the case of $\mathfrak n=\mathds 1_{\mathsf k}$ is also studied in \cite{LY19$_2$}. 
However, this is exactly the case which is completely ignored there. This could be due to two reasons: one is that $\Lambda_\mathsf d(\mathds 1_{\mathsf k}, \mathfrak m)$
is clearly \textit{not} locally faithful, which is the crucial property required in \cite{LY19$_2$}; the other is that this case, at first glance, seems too special. 

\begin{obs}
\label{Obs}
The observations below are obvious and will be used frequently later without any further mention. 
\begin{enumerate}
\item
For every $\mathbf p\in \bN^k$, $\Lambda_{\mathds 1_{\mathsf k}}^{\mathbf p}$ is a singleton: 
$\Lambda_{\mathds 1_{\mathsf k}}^{\mathbf p}=\{\Bx^{\mathbf p}:=\Bx_1^{p_1}\cdots \Bx_{\mathsf k}^{p_{\mathsf k}}\}$. 
Also, $s_\mu$ is a unitary for every $\mu\in\Lambda_{\mathds 1_{\mathsf k}}$. 

\item
$\Lambda(\mathds 1_{\mathsf k},\mathfrak m)$ is pseudo-free. In fact, some computations show 
$a^n|_\mu=a^{n{\mathfrak m}^{d(\mu)}}$ and so $a^n|_\mu =0\iff n=0$. 
\end{enumerate}
\end{obs}

Our next goal is to show that $\Lambda_{\mathsf d}(\mathds 1_{\mathsf k}, \mathfrak m)$ has a canonical Cartan subalgebra. In particular, this canonical Cartan subalgebra is $\F$ if $\{ |m_i|: 1\le i\le \mathsf k\}$
 is multiplicatively independent; it properly contains $\F$, otherwise.

\begin{lem}
\label{L:F}
$\F=\ol{\spn}\{s_\mu u_{a^\ell} s_\mu^*: \mu \in \Lambda_{\mathds 1_{\mathsf k}}, \ell\in \bZ\}=\ol{\spn}\{s_\mu^n u_{a^\ell} s_\mu^{-n}: d(\mu)=\mathds{1_{\mathsf k}}, n\in \bN, \ell\in \bZ\}$. 
\end{lem}

\begin{proof}
It is known and easy to see that 
$\F=\ol\spn\{s_\mu u_{a^m} s_\nu^*\mid \mu,\nu\in \Lambda_{\mathds 1_{\mathsf k}}, d(\mu)=d(\nu), m\in \bZ\}$. 
It follows from Observation \ref{Obs} that $d(\mu)=d(\nu)$ forces $\mu=\nu$,  say equal to $\Bx^{\mathbf n}$ for some $\mathbf n\in \bN^{\mathsf k}$. 
WLOG, we assume that $n_1-n_2=l \ge 0$. Then
\begin{align*}
s_{\Bx^{\mathbf n}} u_{a^m} s_{\Bx^{\mathbf n}}^{-1}
=&s_{\Bx_3^{n_3}\cdots \Bx_{\mathsf k}^{n_{\mathsf k}}} s_{\Bx_1^{n_1}} s_{\Bx_2^{n_2}}  u_{a^m} s_{\Bx_2^{n_2}}^{-1} s_{\Bx_1^{n_1}} ^{-1} s_{\Bx_3^{n_3}\cdots \Bx_{\mathsf k}^{n_{\mathsf k}}}^{-1}\\ 
=&s_{\Bx_3^{n_3}\cdots \Bx_{\mathsf k}^{n_{\mathsf k}}} s_{\Bx_1^{n_1}} s_{\Bx_2^{n_2}} s_{\Bx_2^\ell} u_{a^m|_{\Bx_2^\ell}} s^{-1}_{\Bx_2^\ell} s_{\Bx_3^{n_3}\cdots \Bx_{\mathsf k}^{n_{\mathsf k}}}^{-1} s_{\Bx_2^{n_2}}^{-1} s_{\Bx_1^{n_1}}^{-1}\\
=&s_{\Bx_3^{n_3}\cdots \Bx_{\mathsf k}^{n_{\mathsf k}}} s_{\Bx_1^{n_1}} s_{\Bx_2^{n_1}} u_{a^m|_{\Bx_2^\ell}}  s_{\Bx_3^{n_3}\cdots \Bx_{\mathsf k}^{n_{\mathsf k}}}^{-1} s_{\Bx_2^{n_1}}^{-1} s_{\Bx_1^{n_1}}^{-1}.
\end{align*}
After repeating this process, all the $\Bx_i$'s will have the same exponent. 
\end{proof}

\begin{lem}
\label{L:FAbe}
$\F$ is commutative. 
\end{lem}

\begin{proof}
Compute 
\begin{align*}
(s_\mu u_{a^l} s_\mu^*)(s_\nu u_{a^n} s_\nu^*)
&=s_\mu u_{a^l} s_\nu s_\mu^* u_{a^n} s_\nu^*\ (\text{as $s_\mu^*s_\nu= s_\nu s_\mu^*$})\\
&=s_\mu s_\nu  u_{a^l|_{\nu}} (u_{a^{-n}} s_\mu)^* s_\nu^*\ (\text{as $a^l\cdot \nu =\nu$})\\
&=s_\mu s_\nu u_{a^l|_{\nu}}  (u_{a^{-n}}|_{\mu})^{-1} s_\mu^* s_\nu^*\\
&=s_\mu s_\nu u_{a^l|_{\nu}}  u_{a^n|_{a^{-n}\cdot \mu}} s_\mu^* s_\nu^*\\
&=s_\mu s_\nu u_{a^l|_{\nu}}  u_{a^n|_\mu} s_\mu^* s_\nu^*\ (\text{as $a^{-n}\cdot \mu =\mu$})\\
&=(s_\nu u_{a^n} s_\nu^*)(s_\mu u_{a^l} s_\mu^*). 
\end{align*}
It now follows from Lemma \ref{L:F} that $\F$ is commutative. 
\end{proof}

\begin{lem}
\label{L:F'}
 $\F'=\ol{\spn}\{s_\mu u_{a^n} s_\nu^*:\mu,\nu\in\Lambda_\mathds 1 \text{ with } \mathfrak m^{d(\mu)}=\mathfrak m^{d(\nu)},\ n\in \bZ\}$.
 \end{lem}

\begin{proof} 
Similar to \cite{Yang1},\footnote{The details will be included in the first author's PhD thesis.} it suffices to show $x\in \ran \Phi_{\mathbf p}\cap\F'$ has the given form. 
By the Cuntz-Krieger relation, one could assume that $x=s_\mu A s_\nu^*$ with $d(\mu)-d(\nu)=\mathbf p$ and $A\in \ca(u_a)$. One could further assume that 
$A=f(u_a)$, where $f(z)=\sum_{i=1}^{n} \lambda_{i} z^{M_i}\ne 0$ for some $0\ne \lambda_i\in \bC$ and $M_i\in \bZ$. 

Also, it is clear that $s_\mu s_\nu= s_\nu s_\mu$ and $\ s_\mu s_\nu^*= s_\nu^* s_\mu$ for all $\mu, \nu\in \Lambda_{\mathds 1_{\mathsf k}}$. 
Then, for all $N\in \bZ$ and $\omega\in \Lambda_{\mathds 1_{\mathsf  k}}$, one has 
\begin{align*}
&\ s_\mu A s_{\nu}^*s_{\omega}u_{a^N} s_{\omega}^* - s_{\omega} u_{a^N} s_{\omega}^*s_\mu A  s_{\nu}^*\\ 
=& \ s_\mu \left( \sum_{i=1}^{n} \lambda_i u_{a^{M_i}} \right)  s_\omega s_{\nu}^* u_{a^N} s_{\omega}^* 
 - s_{\omega} u_{a^N} s_\mu s_{\omega}^* \left( \sum_{i=1}^{n} \lambda_i u_{a^{M_{i}}} \right)  s_{\nu}^* \\
= &\ s_\mu s_\omega \left( \sum_{i=1}^{n} \lambda_{i} u_{a^{M_{i}\mathfrak m^{d(\omega)}}} \right) u_{a^{N\mathfrak m^{d(\nu)}}}  s_{\nu}^* s_{\omega}^* 
 - s_{\omega} s_\mu  u_{a^{N\Bm^{d(\mu)}}} \left( \sum_{i=1}^{n} \lambda_{i} u_{a^{M_{i} \mathfrak m^{d(\omega)}}} \right) s_{\omega}^*  s_{\nu}^* \\
= &\ s_{\mu}s_{\omega} \left( \sum_{i=1}^{n} \lambda_{i} u_{a^{M_{i}\mathfrak m^{d(\omega)}}} \right) \left(u_{a^{N\mathfrak m^{d(\nu)}}}   
 -  u_{a^{N{\mathfrak m}^{d(\mu)}}}\right) s_{\omega}^*  s_{\nu}^*\\
 = &\ s_{\mu}s_{\omega} f\big(u_{a^{\mathfrak m^{d(\omega)}}} \big)\left(u_{a^{N\mathfrak m^{d(\nu)}}}   
 -  u_{a^{N{\mathfrak m}^{d(\mu)}}}\right) s_{\omega}^*  s_{\nu}^*.
\end{align*}
After identifying $\ca(a)$ with $\rC(\bT)$ (see Proposition \ref{P:genO}), the above is equal to $0$ iff 
\[
 f\big(z^{\mathfrak m^{d(\omega)}} \big)\left(z^{N\mathfrak m^{d(\nu)}}   
 -  z^{N{\mathfrak m}^{d(\mu)}}\right)=0,
\]
and therefore, if and only if  $\mathfrak m^{d(\mu)}=\mathfrak m^{d(\nu)}$. 
\end{proof}

\begin{lem}
\label{L:F'MASA}
$\F'$ is a MASA of $\O_{\Lambda(\mathds 1_{\mathsf k}, \mathfrak m)}$. 
\end{lem}

\begin{proof}
We first show that $\F'$ is abelian. For this, let $A:=s_\mu u_{a^M} s_{\nu}^*$ and $B:=s_\alpha u_{a^N} s_\beta ^*$ be two standard generators in $\F'$. Then 
\begin{align*}
AB=s_\mu u_{a^M} s_{\nu}^*s_\alpha u_{a^N} s_\beta^* 
&=s_\mu u_{a^M} s_\alpha s_{\nu}^* u_{a^N} s_\beta^* 
=s_\mu s_\alpha u_{a^{M\Bn^{d(\alpha)}}}u_{a^{N\Bm^{d(\nu)}}} s_{\nu}^*s_\beta^*, \\ 
BA=s_\alpha u_{a^N} s_\beta^*  s_\mu u_{a^M} s_{\nu}^*
&=s_\alpha  u_{a^N} s_\mu s_{\beta}^* u_{a^M} s_\nu^* 
=s_\mu s_\alpha u_{a^{N\Bn^{d(\mu)}}}u_{a^{M\Bm^{d(\beta)}}} s_{\nu}^*s_\beta^*. 
\end{align*}
But $\mathfrak m^{d(\mu)}=\mathfrak m^{d(\nu)}$ and $\mathfrak m^{d(\alpha)}=\mathfrak m^{d(\beta)}$ as $A, B\in \F'$. Thus $AB=BA$ and so $\F'$ is abelian. 

Now we show $\F'$ is a MASA. 
Let $s_\alpha u_{a^N} s_{\beta}^*\in \F'$ and $s_\mu A s_{\nu}^*\in \ran\Phi_{\mathbf p}\cap \F''$ with $A\in \ca(u_a)$. Similar to the proof of Lemma \ref{L:F'}, we have for all 
$\mu,\nu\in\Lambda_{\mathds 1_{\mathsf k}}$ and $N\in \bZ$
\begin{align*}
&\ s_\mu A s_{\nu}^*s_{\alpha} u_{a^N} s_{\beta}^* -s_{\alpha} u_{a^N} s_{\beta}^*s_\mu A  s_{\nu}^*\\
=& \ s_\mu \left( \sum_{i=1}^{n} \lambda_{i} u_{a^{M_{i}}} \right)  s_\alpha s_{\nu}^* u_{a^N} s_{\beta}^* 
 - s_{\alpha} u_{a^N} s_\mu s_{\beta}^* \left( \sum_{i=1}^{n} \lambda_{i} u_{a^{M_{i}}} \right)  s_{\nu}^* \\
=&\ s_\mu s_\alpha \left( \sum_{i=1}^{n} \lambda_{i} u_{a^{M_{i}\mathfrak m^{d(\alpha)}}} \right) u_{a^{N\mathfrak m^{d(\nu)}}}  s_{\nu}^* s_{\beta}^* 
 - s_{\alpha} s_\mu  u_{a^{N\mathfrak m^{d(\mu)}}} \left( \sum_{i=1}^{n} \lambda_{i} u_{a^{M_{i} \mathfrak m^{d(\beta)}}} \right) s_{\beta}^*  s_{\nu}^* \\
=&s_{\alpha} s_\mu \left[\left( \sum_{i=1}^{n} \lambda_{i} u_{a^{M_{i}\mathfrak m^{d(\alpha)}}} \right) u_{a^{N\Bm^{d(\nu)}}}  
 -  u_{a^{N\mathfrak m^{d(\mu)}}} \left( \sum_{i=1}^{n} \lambda_{i} u_{a^{M_{i} \mathfrak m^{d(\beta)}}} \right)\right] s_{\beta}^*  s_{\nu}^*. 
\end{align*}
Identify $\ca(u_a)$ with $\rC(\bT)$ and notice $ m^{d(\alpha)}= m^{d(\beta)}$. Then the above is equal to $0$, iff 
\[
f\big(z^{\mathfrak m^{d(\alpha)}}\big)\left(z^{N\mathfrak m^{d(\nu)}} - z^{N\mathfrak m^{d(\mu)}}\right)=0 \iff \mathfrak m^{d(\mu)}=\mathfrak m^{d(\nu)}. 
\]
Hence $s_\mu A s_{\nu}^*\in \F'$ and therefore $\F'=\F''$. 
\end{proof}

Now suppose that $\{|m_i|: 1\le i\le k\}$ is multiplicatively independent. Then $\F'=\F$, which is also the diagonal subalgebra of $\O_{\Lambda(\mathds 1_{\mathsf k},\mathfrak m)}$. So there is a conditional expectation $\Phi$ from $\O_{\Lambda(\mathds 1_{\mathsf k},\mathfrak m)}$ onto $\F$. 
Therefore $\F$ is a Cartan subalgebra of $\O_{\Lambda(\mathds 1_{\mathsf k},\mathfrak m)}$.  

\begin{cor}
\label{C:FCartan}
Suppose that $\{ |m_i|: 1\le i\le \mathsf k\}$ is multiplicatively independent. Then $\F$ is a Cartan subalgebra of $\O_{\Lambda(\mathds 1,\mathfrak m)}$.  
\end{cor}

Suppose that $m_i>0$ $(1\le i\le \mathsf k$). For convenience, we use the convention: If $a|_\mu=a^n$, then $t^{\ln a|_\mu}:=t^{\ln n}$. Thus
\begin{align*}
\ln a|_{\mu\nu}
=\ln \mathfrak m^{d(\mu\nu)}
=\ln \mathfrak m^{d(\mu)}+\ln\mathfrak m^{d(\nu)}=\ln a|_\mu+\ln a|_\nu.
\end{align*}
Therefore we obtain action $\alpha$ of $\bT$ on $\O_{\Lambda(\mathds 1_{\mathsf k},\mathfrak m)}$ as follows:
\[
\alpha_t(s_\mu)=t^{\ln a|_\mu} s_\mu, \ \alpha_t(u_a)=u_a \text{ for all }\mu \in \Lambda_{\mathds 1_{\mathsf k}}\text{ and }t\in \bT.
\]
Define 
\[
\Psi(x):=\int_{\bT} \alpha_t(x) d t\qforal x\in \O_{\Lambda(\mathds 1,\mathfrak m)}.
\]

\begin{lem}
Suppose that $m_i>0$ $(1\le i\le \mathsf k$). Then 
$\F'$ is the fixed point algebra of $\alpha$. Furthermore, $\Psi$ a faithful conditional expectation from $\O_{\Lambda(\mathds 1_{\mathsf k},\mathfrak m)}$ onto $\F'$. 
\end{lem}

\begin{proof}
We only need to show the faithfulness of $\Psi$ here, as other parts can be proved similarly to the corresponding parts for $\Phi_{\mathbf 0}$. 
Let $\Phi_{\mathbf 0}|_{\F'}$ be the restriction of $\Phi_{\mathbf 0}$ onto $\F'$, 
and $\Psi$ be the expectation induced from $\alpha$ above. Then one can check that $\Phi_{\mathbf 0}=\Phi|_{\F'}\circ \Psi$. The faithfulness of 
$\Psi$ follows from that of $\Phi_{\mathbf 0}$. 
\end{proof}

\begin{thm}
\label{T:F'Cartan}
Suppose that $m_i>0$ $(1\le i\le \mathsf k$). 
Then $\F'$ is a Cartan subalgebra of $\O_{\Lambda(\mathds 1_{\mathsf k},\mathfrak m)}$. 
\end{thm}

\begin{proof}
It remains to show that $\F'$ is regular, 
For this, let $A:=s_\mu u_a^M s_{\nu}^*\in \F'$ and $B:=s_\alpha u_a^N s_\beta ^*\in \O_{\Lambda(\mathds 1_{\mathsf k},\mathfrak m)}$. Then 
\begin{align*}
B^*AB
&=s_\beta u_{a^{-N}} s_\alpha^* s_\mu u_{a^M} s_{\nu}^*s_\alpha u_{a^N} s_\beta^* \\
&=s_\beta u_{a^{-N}} s_\mu s_\alpha^* u_{a^M} s_\alpha s_{\nu}^* u_{a^N} s_\beta^* \\
&=s_\beta s_\mu u_{a^{-N\mathfrak m^{d(\mu)}}}  u_{a^{M\mathfrak m^{d(\alpha)}}} s_\alpha^* s_\alpha s_{\nu}^* u_{a^N} s_\beta^* \\
&=s_\beta s_\mu u_{a^{-N\mathfrak m^{d(\mu)}}}  u_{a^{M\mathfrak m^{d(\alpha)}}} u_{a^{N\mathfrak m^{d(\nu)}}}  s_{\nu}^*  s_\beta^* \\
&=s_\beta s_\mu u_{a^{M\mathfrak m^{d(\alpha)}}} s_\nu^* s_\beta^* \text{ (as $\mathfrak m^{d(\mu)}=\mathfrak m^{d(\nu)}$)}\\
&\in \F'.
\end{align*}
Therefore $\F'$ is regular. 
\end{proof}

Let $G$ be a discrete (countable) group.  A subgroup $S\subseteq G$ is called \textit{immediately centralizing} if, for every $g\in G$, we either have 
$\{xgx^{-1}: x\in S\}=\{g\}$ or $\{xgx^{-1}: x\in S\}$ is infinite.\footnote{This definition is slightly different from the one used in \cite{DGNRW20}, but mentioned in  \cite{Duw24} and \cite{DGN21}.
Thanks to Anna Duwenig and Rachael Norton for some discussion.}

\begin{thm}
\label{T:F'Cargen}
$\F'$ is a Cartan subalgebra of $\O_{\Lambda(\mathds 1_{\mathsf k},\mathfrak m)}$.
\end{thm}

\begin{proof}
Let $G:=\langle a, \Bx_i: a\Bx_i=\Bx_i a^{m_i}, 1\le i\le \mathsf k\rangle$. Then it it easy to see that 
 $\O_{\Lambda(\mathds 1_{\mathsf k},\mathfrak m)}\cong \ca(G)$ via $u_a\mapsto a$ and $s_{\Bx_i} \mapsto \Bx_i$.
 By \cite{LY21}, $\O_{\Lambda(\mathds 1_{\mathsf k},\mathfrak m)}$ is amenable and so is $\ca(G)$. Thus  
 $\O_{\Lambda(\mathds 1_{\mathsf k},\mathfrak m)}\cong \ca(G)\cong \ca_{\text r}(G)$.
 
 Let $S:=\{\mu a^n \nu^{-1}: \mathfrak m^{d(\mu)}=\mathfrak m^{d(\nu)}, n\in \bZ\}$. Similar to the proof of Theorem \ref{T:F'Cartan}, one can easily show that $S$ is 
 a normal subgroup of $G$. Also, analogous to the proof of Lemma \ref{L:F'MASA},\footnote{Actually, it is easier here since $A=a^n$ for some $n\in \bZ$.}
 one can show that, for $\alpha a^k \beta^{-1}\in G$, if the set
 $\{(\mu a^n \nu^{-1})(\alpha a^k \beta^{-1})(\nu a^{-n} \mu^{-1}): \mu, \nu \in \Lambda(\mathds 1_{\mathsf k},\mathfrak m), n\in \bZ\}$ is not a singleton, 
 then it has to be infinite. Indeed, 
\begin{align*}
(\mu a^n \nu^{-1})(\alpha a^k \beta^{-1})(\nu a^{-n} \mu^{-1}) = (\mu' a^{n'} \nu'^{-1})(\alpha a^k \beta^{-1})(\nu' a^{-n'} \mu'^{-1})
\iff n \mathfrak{m}^{d(\mu')}=n' \mathfrak{m}^{d(\mu)}.
\end{align*}
Hence $S$ is immediately centralizing. 
By \cite[Theorem 3.1]{DGNRW20}, $\ca_{\text r}(S)$ is Cartan in $\ca_{\text r}(G)$. Therefore $\F'$ is Cartan in 
 $\O_{\Lambda(\mathds 1_{\mathsf k},\mathfrak m)}$.
\end{proof}

\begin{rem}
Notice that,  since $\Lambda_{\mathds 1_{\mathsf k}}$ has a unique infinite path, every triple $(\mu, a^n, \nu)$ is cycline. So the cycline subalgebra of 
$\O_{\Lambda(\mathds 1_{\mathsf k},\mathfrak m)}$ coincides with $\O_{\Lambda(\mathds 1_{\mathsf k},\mathfrak m)}$, and
does not provide much information of the canonical Cartan subalgebra $\F'$ in general. 
\end{rem}

\subsection{The spectrum of $\F$}
\label{SS:end}

We end this paper by computing the spectrum of $\F$ to connect with Furstenberg’s $\times p, \times q$ conjecture in the viewpoint of \cite{BS24, HW17}. 
Let $1\le p_1, \ldots, p_n\in \bN$ and 
\[
\varphi: \bT\to \bT, \ z\mapsto z^{p_1\cdots p_n}.
\] 
Then the inverse limit 
\[
 \varprojlim(\bT, \varphi):=\left\{(x_n)_{n\in \bN}\in \prod_{n\in \bN} \bT: x_n =\varphi(x_{n+1}) \text{ for all }n \in \bN\right\}
\]
is a solenoid, denoted as $S_{p_1\cdots p_n}$. 

Let $\mathfrak{e}:=\Bx_1\cdots \Bx_{\mathsf k}$, the unique path in $\Lambda_{\mathds 1_{\mathsf k}}$ of degree $\mathds 1_{\mathsf k}$. 
For $n \in \bN$, we have 
\begin{align*}
\F_{n\mathds 1_{\mathsf k}}&= \overline{\text{span}}\{s_{\mathfrak{e}^n} u_{a^m} s_{\mathfrak{e}^n}^* | \ m \in \bZ \},\\
\F &= \overline{\textnormal{span}}\{s_{\mathfrak{e}^n} u_{a^m} s_{\mathfrak{e}^n}^* | \ n \in \bN, m \in \bZ \}=\ol{\bigcup_{n\in \bN}\F_{n\mathds 1_{\mathsf k}}}.
\end{align*}
Note that for any $n \in \bN$, $s_{\mathfrak{e}^n} u_a s_{\mathfrak{e}^n}^*$ is a unitary that generates $\F_{n\mathds 1_{\mathsf k}}$. Then one can see that there is an isomorphism 
$\psi_n:\F_{n\mathds 1_{\mathsf k}} \to \rC(\bT)$.

Let $\mathfrak M:=\prod_{i=1}^{\mathsf k} m_i$. 

\begin{prop}
Consider the subalgebras $\F_{n\mathds 1_{\mathsf k}}$ along with the inclusions $\varphi_n: \F_{n\mathds 1_{\mathsf k}} \hookrightarrow \F_{(n+1)\mathds 1_{\mathsf k}}$ 
given by $\varphi_n(s_{\mathfrak{e}^n} u_{a^p} s_{\mathfrak{e}^n}^*) = s_{\mathfrak{e}^{n+1}} u_{a^{p\mathfrak m}} s_{\mathfrak{e}^{n+1}}^*$. Then 
$\F = \varinjlim\limits_{n\to \infty} \F_{n\mathds 1_{\mathsf k}}$, and the spectrum of $\F$ is homeomorphic to $S_{\mathfrak M}$.
\end{prop}

\begin{proof}
Notice that 
\[
s_{\mathfrak{e}^n} u_{a^p} s_{\mathfrak{e}^n}^* 
= s_{\mathfrak{e}^n} s_{\mathfrak{e}}s_{\mathfrak{e}}^* u_{a^p} s_{\mathfrak{e}^n}^* 
= s_{\mathfrak{e}^{n+1}} u_{a^{p\mathfrak M}} s_{\mathfrak{e}^{n+1}}^*.
\]
This shows that $\varphi_n$ is indeed an inclusion of C*-algebras. Since the union $\bigcup_{n=0}^{\infty} \F_{n\mathds 1_{\mathsf k}}$ is dense in $\F$ (actually they are equal), we have that $\F$ is isomorphic to the direct limit of $(\F_{n\mathds 1_{\mathsf k}},\varphi_n)$ (see \cite{Mur04} Remark 6.1.3).

Due to \cite[Theorem 2]{Tak55}, the spectrum of $\F$ is the projective limit of the spectra of the subalgebras $\F_{n\mathds 1_{\mathsf k}}$, with the maps $\phi_{n+1}:\widehat{\F_{(n+1)\mathds 1_{\mathsf k}}} \rightarrow \widehat{\F_{n\mathds 1_{\mathsf k}}}$ which induce the maps $\varphi_n$. However we will work with an isomorphic direct system in order to make things more concrete. Observe that we have an isomorphism $\psi_n$ between $\F_{n\mathds 1_{\mathsf k}}$ and $\rC(\bT)$ that sends the element $s_{\mathfrak{e}^n} a s_{\mathfrak{e}^n}^* \in \F_n$ to the function $f(z) = z$ in $\rC(\bT)$, which we will just denote by $z$ (note that $z$ is a unitary element that generates $\rC(\bT)$ ). Let $\varphi'_{n}: \rC(\bT) \rightarrow \rC(\bT)$ be the map defined by sending the function $z$ to $z^{\mathfrak M}$. Then by a direct calculation the following diagram commutes:
\[ 
\begin{tikzcd}
\F_{n\mathds 1_{\mathsf k}} \arrow{r}{\varphi_{n}} \arrow[swap]{d}{\psi_n} & \F_{(n+1)\mathds 1_{\mathsf k}} \arrow{d}{\psi_{n+1}} \\%
\rC(\bT) \arrow[r, "\varphi'_{n}"']& \rC(\bT)
\end{tikzcd}
\]
By \cite[Proposition 2]{Tak55}, $\F$ is then isomorphic to the direct limit of the system $(\rC(\bT), \varphi'_{n,n+1})$. We now observe that homeomorphism $\rho:\bT \rightarrow \bT$ defined by $\rho(z) = z^{\mathfrak M}$ induces the maps $\varphi'_n:\rC(\bT)\rightarrow \rC(\bT)$, and so the spectrum $\F$ is homeomorphic to the projective limit of
$$ \bT \xleftarrow{\rho} \bT \xleftarrow{\rho} \bT \xleftarrow{\rho} \bT \xleftarrow{\rho} \ldots $$
which is precisely $S_{\mathfrak M}$. 
\end{proof}

\begin{rem}
One can show that $\O_{\Lambda(\mathds 1_{\mathsf k},\mathfrak m)}\cong \F\ltimes \bZ^{\mathsf k}\cong \rC(S_{\mathfrak M}) \ltimes \bZ^{\mathsf k}$. In fact, the action of $\bZ^{\mathsf k}$ on $\F$ is given by 
\[
\alpha: \bZ^{\mathsf k}\to \Aut\F, \
\alpha_{\mathbf n}(A)= s_\nu A s_\nu^*,
\]
where $\nu$ is the unique path in $\Lambda_{\mathds 1_{\mathsf k}}$ of degree $\mathbf n$. 
\end{rem}

\smallskip
\subsection*{Acknowledgements}
Some results in this paper were presented at COSy 2023 and the workshop ``Groups and Group Actions" in Thematic Program on Operator Algebras and Applications 
in 2023. The second author is very  grateful to the organizers for the invitations and providing great opportunities to present our results. 
Also, thanks go to Boyu Li for some discussion at the early stage of this paper, 
and the anonymous referee for careful reading.


\end{document}